\documentclass[11pt]{article}
\usepackage{dsfont}
\usepackage{amsmath, amsthm, amssymb, mathrsfs, eufrak}
\usepackage[abbrev,non-sorted-cites]{amsrefs}
\usepackage[all]{xy}
\usepackage{CJK}
\usepackage[abs]{overpic}
\usepackage{grffile}
\usepackage{graphicx}
\usepackage[usenames, dvipsnames]{color}
\newtheorem{thm}{Theorem}[section]
\newtheorem{cor}[thm]{Corollary}
\newtheorem{lem}[thm]{Lemma}
\newtheorem{ex}[thm]{Example}
\newtheorem{definition}[thm]{Definition}
\newtheorem{rmk}[thm]{Remark}

\newtheorem{prop}[thm]{Proposition}
\newtheorem{conj}[thm]{Conjecture}

\begin{document}

\title{Correspondence Theorem between Holomorphic Discs and Tropical Discs on K3 Surfaces}
\author{Yu-Shen Lin}

\maketitle
\section*{Abstract}
 In this paper, we prove that the open Gromov-Witten invariants defined in \cite{L4} on K3 surfaces satisfy the Kontsevich-Soibelman wall-crossing formula. One one hand, this gives a geometric interpretation of the slab functions in Gross-Siebert program. On the other hands, the open Gromov-Witten invariants coincide with the weighted counting of tropical discs. This is an analog of the corresponding theorem on toric varieties \cite{M2}\cite{NS} but on compact Calabi-Yau surfaces.

\section{Introduction}

The celebrated Strominger-Yau-Zaslow conjecture \cite{SYZ} suggested that Calabi-Yau manifolds admit special Lagrangian fibration near large complex limit points. Furthermore, the special Lagrangians are expected to collapse to the base affine manifolds at the complex structure limit \cite{KS4}. Holomorphic curves under this limit are conjectured to  converge to some $1$-skeletons, known as tropical curves, on the base affine manifolds. Ideally, this can reduce some enumerative problems on Calabi-Yau manifolds to the countings of the tropical curves, which are closer to combinatoric problems. However, the great picture is not generally carried out due to the involvement of Calabi-Yau metric, yet has no explicit expression so far. Instead, Mikhalkin \cite{M2} started the realm of tropical geometry by establishing the correspondence between the counting of holomorphic curves and tropical curves on toric surfaces. Later, the correspondence theorem is generalized to all toric manifolds \cite{NS}.

For understanding mirror symmetry from the SYZ point of view, the author defined an version of open Gromov-Witten invariant $\tilde{\Omega}(\gamma;u)$ \cite{L4}, naively count the number of holomorphic discs with boundary on special Lagrangian torus fibre $L_u$ in relative class $\gamma$. The open Gromov-Witten invariant 
$\tilde{\Omega}(\gamma;u)$ locally is an 
invariant with respect to the special Lagrangian boundary condition $u$. There are some real codimension one walls on the base of the fibration known as the walls of marginal stability.  
An explicit example is constructed that the special Lagrangian torus fibres can bound new holomorphic discs when the boundary conditions vary across a wall of marginal stability and the phenomenon only occurs on one side of the wall. As a consequence, the counting invariant $\tilde{\Omega}(\gamma;u)$ jumps when the special Lagrangian boundary condition varies across the wall. This is similar to the wall-crossing phenomenon of generalized Donaldson-Thomas invariants \cite{KS2}. Under certain primitive conditions, the wall-crossing formula is calculated in \cite{L4}. 

On the other hand, the author developed the notion of tropical discs and the corresponding counting invariants on elliptic K3 surfaces \cite{L4}. The tropical discs counting invariants satisfy the Kontsevich-Soibelman wall-crossing formula. In particular, a relative class can be realized as a tropical discs when the associated open Gromov-Witten invariants are nonzero. The main goal of the paper to show that the open Gromov-Witten invariants is indeed the same as the weighted tropical discs counting. 

%

The paper is arranged as follows: In section two, we review the homological algebra of curved $A_{\infty}$ structures including the notion of Maurer-Cartan equations and their properties. In section three, we define the tropical discs counting as weighted count of tropical discs and study its properties. The generating functions of the tropical discs counting are related to the slab functions in Gross-Siebert program \cite{GS1} or Kontsevich-Soibelman transformation in \cite{KS1}. In particular, they satisfy the Kontsevich-Soibelman wall-crossing formula.

In Section \ref{4006}, we first review the Floer theory on HyperK\"ahler manifolds with special Lagrangian fibration. The hyperK\"ahler condition makes the special Lagrangian fibres which bound holomorphic discs project to affine lines on the base. Then we use the idea of family Floer homology, which is originally advocated by Fukaya \cite{F5}, to study the wall-crossing phenomenon of the open Gromov-Witten invariants. The similar method is also used to  construct the mirror by Tu \cite{T4} and the mirror functor for the Lagrangian torus fibration by Abouzaid \cite{A2}. The result of our paper can be viewed as the enumerative counterpart. More precisely, we study the Maurer-Cartan elements associated to the $A_{\infty}$ structures of the special Lagrangian fibres which can be identified as the cohomology of the torus fibres. When we vary the special Lagrangian boundary conditions in a $1$-parameter family, the $A_{\infty}$ structures are related via pseudo-isotopies. Moreover, pseudo-isotopies naturally identify the Maurer-Cartan elements. The identification records the information of holomorphic discs appearing in the $1$-parameter family and not necessarily coincide with the identification via parallel transport. We match the identification the pseudo-isotopies with the Kontsevich-Soibelman transformation associated to certain tropical discs, including the wall-crossing formula and the computation the contribution from the initial discs. In particular, this leads to the proof of the main theorem of the paper. 
  \begin{thm} (=Theorem \ref{3013})
    Let $\gamma$ be a relative class and $u$ does not fall on the wall of marginal stability of $\gamma$, then the open Gromov-Witten invariants and the weighted tropical discs counting are well-defined and coincide, i.e.
       \begin{align*}
          \tilde{\Omega}(\gamma;u)=\tilde{\Omega}^{trop}(\gamma;u).
       \end{align*}
  \end{thm}

\section*{Acknowledgment}
  The author would like to thank Shing-Tung Yau for constant encouragement and Chiu-Chu Melissa Liu for helpful discussion.  
 The author also wants to thank Kenji Fukaya, Mark Gross, Andy Neitzke, for explaining their earlier work. The author is indebted to the referee(s) for valuable suggestions and comments. 

\section{Preliminary of Algebraic Framework}
 The section is a review of homological algebra in \cite{FOOO}\cite{F1}.
\subsection{$A_{\infty}$-Structures and $A_{\infty}$ Homomorphisms}
   Let $\Lambda^{\mathbb{C}}$ be the complex Novikov ring, namely 
     \begin{align*}
       \Lambda^{\mathbb{C}}=\{\sum_{i=0}^{\infty} a_i T^{\lambda_i}| a_i\in \mathbb{C}, \lambda_i \in \mathbb{C}, \mbox{ and $|\lambda_i|$ increasing}, \lim |\lambda_i|=+\infty \}.
     \end{align*} We will denote its maximal ideal by $\Lambda_+$.
    We will always assume $\mathcal{S}$ is a sector with angle less than $\pi$ (possibly degenerates to a ray). We define 
    \begin{align*}
        \Lambda^{\mathcal{S}}:=\{\sum^{\infty}_{i=0}a_iT^{\lambda_i}\in \Lambda| \mbox{Arg}\lambda_i\in \mathcal{S}, \lim |\lambda_i|=\infty\}.
    \end{align*}      
         There is a natural valuation $val:\Lambda^{\mathbb{C}} (\mbox{or }\Lambda^{\mathcal{S}})\rightarrow \mathbb{R}_{\geq 0}$ given by
                    \begin{align*}
                       val(\sum_i a_i T^{\lambda_i})=|\lambda_0|.
                    \end{align*}  
         We will use $F^{\lambda}\Lambda^{\mathbb{C}}$(or $F^{\lambda}\Lambda^{\mathcal{S}}$) to denote the subset of $\Lambda^{\mathbb{C}}$(or $\Lambda^{\mathcal{S}}$) consists of elements with $val$ less than $\lambda$ for a $\lambda>0$. When the sector $\mathcal{S}$ degenerates to a ray, then $\Lambda^{\mathcal{S}}$ is naturally identified with the standard Novikov ring $\Lambda$. Notice that $val$ gives a filtration on $\Lambda^{\mathbb{C}}$ does not with respect to the multiplication. However, we have the following substitute. 
        \begin{lem} \label{2999}
            Let $\Gamma$ be a lattice and $\mathcal{S}$ be a sector with angle less than $\pi$. Assume that
          $Z:\Gamma \rightarrow \mathbb{C}$ be a homomorphism such that $\Gamma_{\mathcal{S}}:=Z^{-1}(\mathcal{S})$ has finite intersection with $B_0(\lambda)$, for every $\lambda>0$. Then 
          \begin{enumerate}
             \item for every $\lambda>0$, we have $B_0(\lambda)\cap \hat{\Gamma}_{\mathcal{S}}$ is finite, where $\hat{\Gamma}_{\mathcal{S}}$ is the monoid generated by $\Gamma_{\mathcal{S}}$.
             \item given $\lambda'>0$, then there exists $\lambda>0$ such that 
                   and a subset $A$ of $Z(\Gamma_{\mathcal{S}})$ which does not intersect with $B_0(\lambda)$. Then there exists $\lambda'>0$, independent of $A$, such that the monoid generated by $A$ has no intersection with $B_0(\lambda')$. 
          \end{enumerate}
          
          \end{lem}
           \begin{proof}
               It suffices to prove the case when the sector $\mathcal{S}$ has angle $\pi/2$ , since the topology will be equivalent. We will put $\mathcal{S}$ as the first quadrant. For any $\lambda>0$, $\hat{x}\in B_0(\lambda)\cap \hat{\Gamma}_{\mathcal{S}}$ then $\hat{x}=\sum x_i$ with $x_i\in B_0(\lambda)\cap \Gamma_{\mathcal{S}}$. Since the set $B_0(\lambda)\cap \Gamma_{\mathcal{S}}$ is finite, set $\lambda'$ be the smallest $x$ or $y$ coordinate of elements in $B_0(\lambda)\cap \Gamma_{\mathcal{S}}$. Then there are at most $2[\frac{\lambda}{\lambda'}]+2$ of such $\hat{x}$. This proves the first part of the lemma. The second lemma is straight forward after we transform $\mathcal{S}$ to a sector with angle $\pi$.
           \end{proof}

   \begin{definition}
     Let $G$ be a monoid with an evaluation map $\omega: G\rightarrow \mathbb{R}_{\geq 0}$\footnote{In the later application, we will take $\omega$ be either the symplectic form or the central charge defined in (\ref{4010}).} such that 
       \begin{align*}
           |\omega^{-1}[0,\lambda)| < \infty,
       \end{align*} for every $\lambda\in \mathbb{R}_{\geq 0}$. 
   \end{definition} 
   
   Let $C$ be a graded vector space over $\Lambda$ and let 
    \begin{align*}
     BC[1]=\bigoplus_k B_k C[1], \hspace{5mm} B_kC[1]=C[1]^{\otimes k}
     \end{align*} be its associated bar complex. 
  \begin{definition}
     We say that $C$ admits an filtered gapped $A_{\infty}$ algebra structure if there exists a monoid $G$ and a sequence of homomorphisms $\{m_{k,\gamma}\}_{k\geq 0,\gamma\in G}$ of degree $+1$ 
        \begin{align*}
           m_{k,\gamma}: B_kC[1]\rightarrow C[1]
        \end{align*} 
       such that $\hat{d}\circ \hat{d}=0$, where $\hat{d}$ is the coderivation induced from 
         \begin{align*}
            m_k=\sum_{\gamma\in G} m_{k,\gamma} T^{\omega(\gamma)}, k\geq 0.
         \end{align*} 
  \end{definition}

\begin{definition} Let $(C,m_{k,\gamma})$ and $(C',m'_{k,\gamma})$ be two gapped filtered $A_{\infty}$ algebras. The sequence of $\mathbb{R}$-linear maps $f_{k,\gamma}: B_k C[1]\rightarrow C'[1]$ of degree $0$, for $k\in \mathbb{N}$ and $\gamma\in G$ be an $A_{\infty}$ homomorphism if $\hat{d}'\circ \hat{f}=\hat{f}\circ \hat{d}$, where $\hat{f}, \hat{d}$ is the corresponding coderivation. 
\end{definition}

\subsection{Canonical Model and Homotopy Lemma}
     Let $(C,m_k)$ be a filtered $G$-gapped $A_{\infty}$-algebra and
          $m_{1,0}^2=0$. Let $H$ be its $m_{1,0}$-cohomology and the projection $p:C\rightarrow H$ is of degree zero. Assume that there exists a deformation retract from $C$ to $H$, namely, a pair $(H,\iota)$ where 
       \begin{enumerate}
         \item $\iota:H\rightarrow C$ of degree zero is an injective homomorphism of vector spaces such that 
            \begin{align*}
               d\iota=0, pd=0 \mbox{ and } p\iota= \mbox{Id}.
            \end{align*}
         \item The linear map $G:C\rightarrow C$ of degree $-1$ such that 
           \begin{align} \label{3}
              \mbox{Id}-\iota p= -(dG+Gd).
           \end{align}
       \end{enumerate} Then we have the famous homotopy perturbation lemma. 
  \begin{thm}  \label{25}\cite{FOOO}
    Under the assumption above, there exists an $A_{\infty}$ structure on $H$ and an invertible $A_{\infty}$ homomorphism $f:H\rightarrow C$. 
  \end{thm}
 \begin{proof} The proof is standard but we include a sketch of the proof for introducing certain notations used later. We construct the canonical $A_{\infty}$ structure and $A_{\infty}$ homomorphism as follows:
  \begin{definition}
    Let $Gr(k,\gamma)$ be the set of decorated rooted tree $T$ with the following properties:
      \begin{enumerate}
        \item The vertices $C_0(T)$ are the disjoint union of exterior vertices $C^{ext}_0(T)$ and interior vertices $C^{int}_0(T)$. The root $v_0\in C^{ext}_0(T)$ and $|C^{ext}_0(T)|=k+1$.
        \item Each exterior vertex has valency $1$.
        \item Each interior vertex $v\in C^{int}_0(T)$ is labeled by $\gamma_v\in G$. The valency of $v$ is greater than $2$ if $\gamma_v=0$.
        \item The sum of labels of interior vertices is $\gamma$.
      \end{enumerate}
  \end{definition} With the above definition, an edge is called an exterior edge if it is adjacent to an exterior vertex. Otherwise, it is called an interior edge. 

  We associate a homomorphism for each tree $T\in Gr(k,\gamma)$ as follows: For the edge adjacent to the root, we insert the projection operator $p$. For other edges adjacent to an exterior vertex, we insert the operator $\iota$. For each interior edge, we insert the operator $G$. We insert the operator $m_{k_v, \gamma_v}$ for each interior vertex, where $k_v=\mbox{val}(v)-1$. Let $m^{can}_{k,\gamma}$ denotes summation of the composition operator in above way, summing over all possible decorated rooted trees in $Gr(k,\gamma)$. Then it is straight forward to check that $\{m_{k,\gamma}\}$ defines a filtered $G$-gapped $A_{\infty}$ structure on $H$. 
  %
  
  Moreover, if we replace operator associate to the edge adjacent to the root by $G$ then the similar composition operators $f_{k,\gamma}$ define an $A_{\infty}$-homomorphism from $H$ to $C$.  
\end{proof}

\subsection{Pseudo-Isotopies of $A_{\infty}$ Algebras}
\begin{definition}
   Two gapped filtered $A_{\infty}$ algebras $(C,m^0_{k,\gamma})$ and $(C,m^1_{k,\gamma})$ are pseudo-isotopy if there exists an gapped filtered $A_{\infty}$ structure on $\big(C^{\infty}([0,1]_t,C), \hat{m}_{k,\gamma}\big)$ restricting to $(C, m^0_{k,\gamma})$ (and $(C, m^1_{k,\gamma})$) as $t=0$ (and $t=1$ respectively).
\end{definition}
\begin{rmk}
  For $\boldsymbol{x}_i=x_i(t)+dt\wedge y_i(t)\in C^{\infty}([0,1],C)$, we write
     \begin{align*}
        \hat{m}_k(\boldsymbol{x_1},\cdots,\boldsymbol{x_k})=x(t)+dt\wedge y(t), 
     \end{align*} where
     \begin{align*}
        x(t)=& m_k^t(x_1(t),\cdots,x_k(t)), \\
        y(t)= &\mathfrak{c}^t_k(x_1(t),\cdots,x_k(t)) \\
              &-\sum_{i=1}^k (-1)^* m^t_{k,\gamma}(x_1(t),\cdots, x_{i-1}(t), y_i(t), x_{i+1}(t),\cdots x_k(t)),
     \end{align*} if $(k,\gamma)\neq (1,0)$ and 
     \begin{align*}
         y(t)=\frac{d}{dt}x_1(t)+m^t_{1,0}(y_1(t)).
     \end{align*} Then the $A_{\infty}$ relation of $\hat{m}_k$ is equivalent to the following relations:
     \begin{enumerate}
        \item $m^t_{k,\gamma}$ and $\mathfrak{c}^t_{k,\gamma}$ are smooth in $t$\footnote{For the purpose of this paper, we will have $C=\Omega^*(L)$. and $m^t_{k,\gamma}, c^t_{k,\gamma}$ are smooth differential forms on $L\times [0,1]$.}. 
        \item For any fixed $t$, $(C,m^t_{k,\gamma})$ defines a gapped filtered $A_{\infty}$ structure.
        \item \begin{align*}
            &\frac{d}{dt}m^t_{k,\gamma}(x_1,\cdots,x_k)\\
            &+\sum (-1)^* \mathfrak{c}^t_{k_1,\gamma_1}(x_1,\cdots, m^t_{k_2,\gamma_2}(x_i,\cdots),\cdots,x_k)\\
            &-\sum m^t_{k_1,\gamma}(x_1,\cdots,\mathfrak{c}^t_{k_2,\gamma_2}(x_i,\cdots ),\cdots x_k)=0,            
        \end{align*} where the sum is all possible $i, (k_1,\gamma_1)$ and $(k_2,\gamma_2)$ such that $k_1+k_2=k+1$ and $\gamma_1+\gamma_2=\gamma$.
        \item $\frac{d}{dt}m^t_{k,0}=0$ and $\mathfrak{c}^t_{k,0}=0$ for every $k$.
     \end{enumerate}
\end{rmk}

\begin{thm} \cite{F1} \label{4050}
  If there exists a pseudo-isotopy between two (filtered) $A_{\infty}$ algebras $(C,m^0_{k,\gamma})$ and $(C,m^1_{k,\gamma})$, then there exists an invertible $A_{\infty}$ homomorphism between them. 
\end{thm}
\begin{proof}
 We will need the construction of the $A_{\infty}$ homomorphism for later usage, so here we sketch the construction the $A_{\infty}$ homomorphism.
 
 Let $Gr'(k,\gamma)$ be the set of trees $T$ in $Gr(k,\gamma)$\footnote{Here we assume moreover that all the vertices with valency $1$ are exterior vertices.} together with a time allocation $\tau$ (See \cite{FOOO}\cite{F1}). 
 We will define $\mathfrak{c}(T) $ inductively on the number of exterior vertex of $T$. 
   \begin{enumerate}
     \item If $|C^{int}_0(T)|=0$, we set $\mathfrak{c}(T,\tau)=Id$.
     \item If $|C^{int}_0(T)|=1$, let $v\in C^{int}_0(T)$ and $\mbox{val}(v)=k+1$. Then  we set
        \begin{align*}
         \mathfrak{c}(T,\tau)(x_1,\cdots,x_k)=-\mathfrak{c}^{\tau(v)}_{k,\gamma_v}(x_1,\cdots,x_k) .
        \end{align*}
     \item If $|C^{int}_0(T)|>1$, let $v$ be the vertex adjacent to the root. Assume that $(T_1,\tau_1),\cdots (T_k,\tau_k)$ are the subtrees (with time allocations) derived from deleting $v$ and the root. Then we set 
      \begin{align*}
       \mathfrak{c}(T,\tau)=-\mathfrak{c}^{\tau(v)}_{k,\gamma_v}(\mathfrak{c}(T_1,\tau_1)\otimes\cdots \otimes \mathfrak{c}(T_k,\tau_k) ).
      \end{align*}      
   \end{enumerate}
  
 With the notation above, the $A_{\infty}$ homomorphism from $(C, m^{t_0}_{k,\gamma})$ to $(C,m^{t_1}_{k,\gamma})$ is given by 
  \begin{align*}
     \mathfrak{c}(k,\gamma)=\sum_{Gr(k,\gamma)} \mathfrak{c}(T),
  \end{align*} where $\mathfrak{c}(T)=\int \mathfrak{c}(T,\tau) d\tau$ is the integral over the space of possible time allocation of $T$ with values in $[t_0,t_1]$. We refer the detail of the proof to \cite{F1}. 

\end{proof}


\begin{prop} \cite{F1} \label{20}
  If two (cyclic) gapped filtered gapped $A_{\infty}$ algebra are pseudo-isotopic to each other then so are their canonical models. 
\end{prop}

\subsection{ Maurer-Cartan Equation and Maurer-Cartan Elements}
   Given a filtered gapped $A_{\infty}$ algebra $C$, the evaluation map $\omega:C\rightarrow \mathbb{R}_{\geq 0}$ naturally extends to $C$, which we still denote it by $\omega$. The kernel $\bar{C}=\omega^{-1}(0)\subseteq C$ is a vector space and we have the following definition.
%
  
  \begin{definition}
   \begin{enumerate}
     \item Let $C$ be an gapped filtered $A_{\infty}$ algebra. The Maurer-Cartan equation of $C$ is given by 
             \begin{align*} 
                m(e^b):=m_0(1)+m_1(b)+\cdots +m_k(b,\cdots,b)+\cdots=0.
             \end{align*} We say $b\in C$ is a Maurer-Cartan element if $m(e^b)=0$. We denote $\tilde{\mathcal{MC}}(C)$ to be the set of Maurer-Cartan elements.
     \item Let $b_0,b_1\in \tilde{\mathcal{MC}}(C)$, we say $b_0$ is gauge equivalent to $b_1$ if there exists $b(t)$ and $c(t)$ such that \begin{enumerate}
         \item $b(0)=b_0$ and $b(1)=b_1$.
         \item \begin{align*}
             \frac{d}{dt}b(t)+\sum_k m_k(b(t),\cdots,b(t),c(t),b(t),\cdots,b(t))=0
               \end{align*}
     \end{enumerate}
     \item  The moduli space of Maurer-Cartan elements is defined by \begin{align*}
                        \mathcal{MC}(C):=\tilde{\mathcal{MC}}(C)/\sim,
                     \end{align*} where $\sim$ is the gauge equivalence.      
   \end{enumerate}
        
    \end{definition}
%
%

 \begin{prop} \cite{F1} \label{35}
    If there exists a pseudo-isotopy $m^t_{k,\gamma},\mathfrak{c}^t_{k,\gamma}$ between two filtered $A_{\infty}$ algebras $(C,m^0_{k,\gamma})$ and $(C,m^1_{k,\gamma})$, then induces an isomorphism between Maurer-Cartan spaces
     \begin{align*}
      F:\mathcal{MC}&(C) \longrightarrow \mathcal{MC}(C') \\  
         & b \longmapsto  \sum_{k\geq 1,\gamma} \mathfrak{c}(k,\gamma)(b,\cdots,b)T^{\omega(\gamma)}. 
     \end{align*}   
 \end{prop}

 \begin{thm} \cite{F1} \label{47}
   Let $F_1$ and $F_2$ be two pseudo-isotopies between filtered $A_{\infty}$ algebras $C$ and $C'$. Assume that there exists a pseudo-isotopy of pseudo-isotopy between $F_1$ and $F_2$ then 
      \begin{align*}
         (F_1)_*=(F_2)_*:\mathcal{MC}(C)\rightarrow \mathcal{MC}(C'). 
      \end{align*}
 \end{thm}
 
\begin{ex}
   Let $L\subseteq X$ be a Lagrangian torus which bounds no holomorphic discs. Then the $A_{\infty}$ structure on $\Omega^*(L,\Lambda)$ reduces to a differential graded algebra with 
      \begin{align*}
          m_k= \begin{cases} \pm d & \text{if $k=1$}\\
                      \pm \int_L \cdot \wedge \cdot & \text{if $k=2$} \\ 
                      0& \text{otherwise.} \end{cases}
      \end{align*} The associate Maurer-Cartan space is the set of de Rham equivalent classes of degree $1$. The canonical model of $(\Omega^*(L,\Lambda),m_k)$ reduces to the cohomology ring $(H^*(L,\Lambda),m_k^{can})$. Namely, 
      \begin{align*}
         m_k^{can}= \begin{cases} \pm \langle, \rangle & \text{if $k=2$}\\
                      0& \text{otherwise.} \end{cases}
      \end{align*} The Maurer-Cartan space of $H^*(L,\Lambda)$ is the space of harmonic $1$-forms on $L$ with values in $\Lambda_+$. The induced isomorphism from Theorem \ref{25} between Maurer-Cartan spaces $\mathcal{MC}(\Omega^*(L,\Lambda))$ and $\mathcal{MC}(H^*(L,\Lambda))$ is the identification between  harmonic forms and de Rham classes.
\end{ex}

\section{Recap of HyperK\"ahler Geometry}
   We will have a brief recap of the hyperK\"ahler rotation which we will use through the whole paper. Let $X$ be a complex surface with the holomorphic symplectic $2$-form $\Omega$. Assume that there exists a K\"ahler form $\omega$ in the K\"ahler class $[\omega]$ such that 
      \begin{align*}
         \omega^{2}=\frac{1}{2} \Omega\wedge \bar{\Omega},
      \end{align*} then we say $X$ is a hyperK\"ahler surface.
 \begin{ex}
     Compact holomorphic symplectic manifolds are always hyperK\"ahler by the theorem of Yau \cite{Y1}. There are also non-compact examples such as Ooguri-Vafa space with Ooguri-Vafa metric \cite{OV}.
 \end{ex}     
      The pair $(\omega,\Omega)$ determines an $S^1$-family\footnote{It actually determines an $S^2$-family of hyperK\"ahler structures but we will only use the particular $S^1\subseteq S^2$.} of hyperK\"ahler structures on the underline space of $X$ with the K\"ahler form and holomorphic symplectic $2$-form given by 
         \begin{align*}
             \omega_{\vartheta}:&=-\mbox{Im}(e^{-i\vartheta}\Omega),\\
                 \Omega_{\vartheta}:&=\omega-iRe(e^{-i\vartheta}\Omega).
         \end{align*} for $\vartheta\in S^1$. We will denote the corresponding hyperK\"ahler manifold by $X_{\vartheta}$. 
 
      Let $L_{u_0}$ be a holomorphic curve in $X$, thus $\Omega|_{L_{u_0}}=0$. Then $L_{u_0}$ is a special Lagrangian in $X_{\vartheta}$ in the sense that $\omega_{\vartheta}|_{L_{u_0}}=0$ and $\mbox{Im}\Omega_{\vartheta}|_{L_{u_0}}=0$. This procedure is called the hyperK\"ahler rotation. Let $B$ be the linear system of $L_{u_0}$ and $B_0$ parametrizes the smooth ones. Let $u_0$ be the point in $B_0$ corresponding to $L_{u_0}$ and similarly for each $u\in B_0$, we denote the corresponding smooth holomorphic Lagrangians by $L_u$. After hyperK\"ahler rotation $L_u$ will become a special Lagrangian in $X_{\vartheta}$. Before we can talk about tropical geometry, we have to introduce the integral affine structure on $B_0$, for a given $\vartheta$. Let $e_i\in H_2(X,L_{u_0};\mathbb{Z}), i=1,\cdots, 2g$ be an integral basis. For any $u\in B_0$ near $u_0$ and a path $\phi(t)$ connecting $u_0$ and $u$, the functions 
       \begin{align*}
          f_i(u)=\int_{\bar{e_i}_u} \mbox{Im}\Omega_{\vartheta}, i=1,2
       \end{align*} give an integral affine structure on $B_0$ \cite{H2}. We will use $B_{\vartheta}$ to denote the affine manifold (to distinguish the affine structures constructed from various $\vartheta\in S^1$). Here $\bar{e_i}_u$ is the $2$-chain which is the union of parallel transport of $\partial \gamma$ along the path $\phi$. In the case $n=2$, the affine structure described above is known as the complex affine coordinate in the context of mirror symmetry. The the affine functions $f_i$ are real analytic (multi-value) functions on $B_0$.

\section{Tropical Geometry}
    Naively, tropical curves/discs are union of affine line segments on an integral affine manifold with the "balancing condition".
    In this section, we will explore the tropical geometry on K3 surface (with special Lagrangian fibration and only $I_1$-type singular fibres). It worth mentioning that the tropical curves/discs in this paper live on the base of special Lagrangian fibration. On the other hand, Gross-Siebert \cite{GS1} considered the tropical curves on the dual intersection complex from a toric degeneration (see also \cite{Y4} for the setting in non-Archimedean geometry). However, it is not clear if the bases of the special Lagrangian fibrations can be identified with the affine manifolds in the Gross-Siebert program. This is due to the complexity of the details of Ricci-flat metric on K3 surfaces.  
    It also worth mentioning that there is a tropical discs counting/tropical Donaldson-Thomas invariants defined in \cite{L4}\cite{KS5} by assigning initial conditions and use the expected wall-crossing formula to define the invariants. Here in this paper, we define the tropical discs counting as the weighted count of admissible tropical discs, including the ones passing through the singularities of the affine manifold. 
   
   First we give the definition of tropical discs. See a similar definition in \cite{KS5}.
\begin{definition} \label{322}
    Let $B$ be an affine $2$-manifold with singularities $\Delta$ and with an integral structure on $TB$. In other words, there exists an integral affine structure on $B\backslash \Delta$. Assume that around each singularity of the affine structure the monodromy is conjugate to $\bigl(
    \begin{smallmatrix}
      1 & 1\\
      0 & 1
    \end{smallmatrix} \bigr)$\footnote{We will discuss the tropical geometry with other singularities in the next paper \cite{L10}.}. Let $B_0$ be the complement of the singularities $\Delta$. Let $T$ be a rooted connected tree (with root $x$), we denote the set of vertices and edges by $C_0(T)$ and $C_1(T)$ respectively. A tropical curve (with stop $u$) on $B$ is a $3$-tuple $(\phi,T,w)$ where $T$ is a rooted connected tree (with a root $x$), a weight function $w:C_1(T)\rightarrow \mathbb{N}$ and $\phi:T\rightarrow B$ is a continuous map such that 
  \begin{enumerate}
    \item $\phi(x)=u\in B_0$.
    \item We allow $T$ to have unbounded edges only when $B$ is non-compact.
    \item For any vertex $v\in C_0(T)$, the unique edge $e_v$ closest to the root $x$ is called the outgoing edge of $v$ and write $w_v:=w(e_v)$. 
    \item For each $e\in C_1(T)$, $\phi|_e$ is either an embedding of affine segment on $B_0$ or $\phi |_e$ is a constant map. Assume that $\phi(e)\nsubseteq  \Delta$. Let $v_e\in \Gamma(e,\phi^*T_{\mathbb{Z}}B_0)$ be a primitive flat section. In the former case, we will further assume that $v_e$ is in the tangent direction of $\phi(e)$ and pointing toward $\phi(x)$. 
    \item For each $v\in C_0(T)$, $v\neq x$ and $\mbox{val}(v)=1$, we have $\phi(v)\in \Delta$. Moreover, 
      \begin{enumerate}
         \item If $\phi|_{e_v}$ is an embedding, then $\phi(e_v)$ is in the monodromy invariant direction\footnote{Straight forward computation shows that the monodromy invariant direction of the affine structure is rational.}.
         \item Let $T'\subseteq T$ be a connected subtree contains $v$ such that $\phi(T')=\phi(v)$, then 
            \begin{enumerate}
              \item For $e_v\in C_1(T')$, we associate an primitive integral vector $v_{e_v}\in T_{\phi(v)^+}B$ in the monodromy invariant direction such that $\phi(v)^+=Exp_{\phi(v)}(\epsilon v_{e_v})$ for some $0<\epsilon\ll 1$.
              \item For each edge $e$ in $T'$ or edge adjacent to $T'$, we associate a primitive integral vector $e_v\in T_{\phi(v)^+}B$. 
              \item For each $v\in C_0(T')\cap C_0^{intf}(T)$, the balancing condition (\ref{1099}) holds at $T_{\phi(v)}^+B$.  
              \item Let $e$ be an edge adjacent to $T'$. Then $v_e$ (up to parallel transport along the shortest path from $\phi(v)^+$ to $\phi(e)$) is the primitive tangent vector of $\phi(e)$.   
            \end{enumerate}
      \end{enumerate}
     
    \item For each $v\in C_0(T)$, $v\neq x$ and $\mbox{val}(v)=2$, we have $\phi(v)\in \Delta$. Moreover, 
      \begin{enumerate}
         \item the edges $e_v^+,e_v^-$ adjacent to $v$ are not contracted by $\phi$.
         \item $\phi(e_v^{\pm})$ is in the monodromy invariant direction and the primitive integral vectors associated to $e^{\pm}_{v}$ is monodromy invariant.
         \item $w(e_v^+)=w(e_v^-)$.
      \end{enumerate}
    \item For each $v\in C_0(T)$, $\mbox{val}(v)\geq 3$, $\phi(v)\notin \Delta$, we have the following assumption: 
     (balancing condition) Let $e_1,\cdots, e_n,e_v$ are the edges adjacent to $v$. Then 
          \begin{align} \label{1099}
             w_{e_v}v_{e_v}=\sum_i w_{e_i}v_{e_i}\in T_{\phi(v)}B.
          \end{align}
  \end{enumerate}
\end{definition}
 \begin{rmk} By induction, all $v_e$ are of rational slope.\end{rmk}
The main difference from the usual definition of tropical discs in the literature is that we allow an edge maps to a point and the affine structures can have singularities. The
balancing condition (\ref{1099}) will make the following definition well-defined.
\begin{definition} \label{2000}
    Let $\phi: T\rightarrow B$ be a parametrized tropical rational curves (with stop). Assume that $\mbox{val}(v)=3$ for a vertex $v\neq x$. The multiplicity at such vertex $v$ is given by 
       \begin{equation} \label{1010}
            Mult_v(\phi)=w_1 w_2|v_{e_1} \wedge v_{e_2}|,
       \end{equation}
    where $e_1, e_2$ are two of the edges adjacent to $v$ and $w_i=w(e_i)$. The last term $v_{e_1}\wedge v_{e_2}$ in (\ref{1010}) falls in $\wedge^2 T_{\phi(v)}B\cong \mathbb{Z}$ due to the integral structure of $TB_0$. Here we will replace $\phi(v)$ by $\phi(v)^+$ if $\phi(v)\in \Delta$.
\end{definition}

\begin{definition} \label{53}
  A tropical disc at $u$ is a tropical curve $(w,T,\phi)$ with stop $u\in B_0$ and $T$ is a tree.
\end{definition} Naively, tropical discs with stop at $u$ corresponds to holomorphic discs with boundaries on $L_u$. One very different nature of tropical discs on K3 surfaces is the following observation, which is a direct consequence of Lemma \ref{1911} and reflects the fact that the moduli space
 of holomorphic discs with boundary on special Lagrangian torus fibres have negative virtual dimension.
\begin{prop}
  Every tropical discs in $B$ is rigid.
\end{prop} 

Since our goal is to define a tropical discs counting and compare it with open Gromov-Witten invariants, we first have to associate a relative class for each tropical disc. Recall that there exists a natural complex structure $J_{\vartheta}^{sf}$ on $TB_{\vartheta}$. We may identify the underlying space of $X$ with singular fibres deleted  with $TB_{\vartheta}/\Lambda$, on which there is a nowhere vanishing holomorphic $(2,0)$-form $\Omega_{\vartheta}^{sf}$, for some lattice $\Lambda$. Fix $\vartheta\in S^1$ and a singularity $u_0\in \Delta$, there exist two rays $l_{\vartheta}^{\pm}$ emanating from $u_0$ such that the tangents are in the monodromy invariant directions. Let $u\in l_{\vartheta}^{\pm}$ and $v\in T_{u}B_{\vartheta}$ be the tangent of $l_{\vartheta}^{\pm}$. Let $\tilde{u}$ be the point on the zero section of $TB_{\vartheta}/\Lambda$ which project to $u\in l_{\vartheta}^{\pm}$. Let $\tilde{v}\in T_{\tilde{u}}X$ be a unique lifting of $v$ by identifying the zero section with $B_{\vartheta}$.
Then $J_{\vartheta}^{sf}\tilde{v}$ gives a vector field along the fiber of $T_uB_{\vartheta}/\Lambda$. Since $v$ is rational, there exists a closed geodesic (with orientation) with respect to the flat metric on $L_u$ and tangent $J\tilde{v}$.    
\begin{definition} \label{303}
  Let $X\rightarrow B$ be an elliptic fibration with $24$ singular fibres. Fix $e^{i\vartheta}\in S^1$ and let $B_{\vartheta}$ be the affine manifold with the affine structure given by the complex affine structure with singularities of $X_{\vartheta}$.
   \begin{enumerate}
   \item
   Given a tropical disc $\phi:T\rightarrow B$ with stop at $u$ on the
   integral affine manifold $B_{\vartheta}$, we will associate it with a relative class as
   follows by induction on the number of singularities of affine structure $\phi$
   hits: If the $\phi$ only hits only one singularity and has its stop at $u$, then let
   $[\phi]\in H_2(X,L_u)$ be the relative class of Lefschetz thimble
   such that with $\partial [\phi]$ is the closed geodesic described above \footnote{One may also use the auxiliary data $\omega_{\vartheta}$ and determine the sign by $\mbox{Arg}\int_{[\phi]}\Omega>0$}. Otherwise let $p$
   be the internal vertex of $\phi$ closest to the root $x$ and let $\phi_1,\cdots
   \phi_s$ be the components of $\mbox{Im}(\phi)\backslash p$
   containing an ingoing edge of $p$. By induction we already define $[\phi_i]\in H_2(X,L_p)$. Then $[\phi]$ is defined to be the parallel transport of $\sum_{i=1}^s[\phi_i]$ from $\phi(p)$ to $u$.
   \item Given a holomorphic volume form $\Omega$ of $X$. Then  the central charge of a tropical disc $(\phi,T,w)$  on $B_{\vartheta}$ is defined to be
   \begin{align*}
   Z_{\phi}=\int_{[\phi]}\Omega.
   \end{align*}
   \end{enumerate}
\end{definition}

 \begin{definition}
    \begin{enumerate}
       \item Let $v_1,\cdots, v_n\in M\cong \mathbb{Z}^2$ be primitive vectors (not necessarily distinct) and $\mathfrak{d}_{ij}\in M\otimes \mathbb{R}$ be the lines in the direction $v_i$ with weight $w_{ij}$, $j=1,\cdots, l_i$. Assume that $w_{ij}\leq w_{ij'}$ if $j\leq j'$. We order $\mathfrak{d}_{ij}$ such that $\mathfrak{d}_{i_1j_1}<\mathfrak{d}_{i_2j_2}$ if 
        \begin{enumerate}
            \item $i_1<i_2$ or
            \item $i_1=i_2$ and $j_1<j_2$. 
        \end{enumerate}
       \item  We say that the lines $\{\mathfrak{d}_{ij}\}$ are in the standard position if the intersection of $\mathfrak{d}_{i_1j_1}$ and $\mathfrak{d}_{i_2j_2}$ is on the far right side of the line $\mathfrak{d}_{ij}$ if $\mathfrak{d}_{ij}>\mathfrak{d}_{i_1j_1}$ and $\mathfrak{d}_{ij}>\mathfrak{d}_{i_2j_2}$. 
       \item   Let $(\phi,T,w)$ be a tropical curve in $\mathbb{R}^2$ with no contracted edges. We say $(\phi,T,w)$ is in the standard position with respect to $\{(v_i,w_{ij})\}$ if $T$ has $\sum_i l_i+1$ unbounded edges such that all except one unbounded edges are mapped into some $\mathfrak{d}_{ij}$ with weight $w_{ij}$ by $\phi$. The exceptional unbounded edge has direction $v$ and weight $w$ such that $wv=\sum_{i,j} v_i w_{ij}$. 
            
    \end{enumerate}
    \end{definition}
    The following definition explains which tropical discs will contribute to the tropical discs counting invariants in Definition \ref{1011}. 
    \begin{definition} \label{1012}
       A tropical disc $(\phi,T,w)$ with stop $u\in B_0$ is called an admissible tropical disc if the following holds:
       \begin{enumerate}
          \item For every vertex $v\in C_0(T)$, its valency $\mbox{val}(v)\leq 3$.
          \item Assume that $e\in C_1(T)$ is contracted to a point $\phi(e)\in B$\footnote{If $\phi(e)\in \Delta$, then we will use $\phi(e)^+$ instead of $\phi(e)$.}. The preimage of $\phi(e)$ is a disjoint union of subtrees of $T$. Let $T_e$ be the connected subtree containing $e$. Let $e_0,e_1,\cdots,e_m \in C_1(T)$ be the edges adjacent to $T_e$ and $e_0$ is the one closest to the root. Denote the weight of $e_i$ by $w_i$. Let $T'$ be the tree obtained by adding edges $e_0,\cdots, e_m$ with weight $w_1,\cdots w_m$ and $\tilde{T}$ be the tree by replacing each $e_i\in C_1(T')$ by an unbounded edge with weight $w_i$. $T\backslash T'$ is a disjoint union of subtrees of $T$. Let $T_i$ be the connected subtree containing $e_i$, $i=1,\cdots m$. Then $\phi_i=(\phi|_{T_i}, T_i, w|_{T_i})$ defines a tropical disc with stop at $\phi(e)$.

           For each $e_i$ there exists a primitive relative class $[\phi_i]\in H_2(X,L_{\phi(e)})$.
           Let $v_i\in T_{\phi(e)}B_0$ be the primitive vector such that $v_i|Z_{[\phi_i]}|>0$ and $v_i\mbox{Arg}Z_{[\phi_i]}=0$. Let $(v_i,Z_{[\phi_i]}), i=1,\cdots,n$ be such distinct pairs and with the order such that 
              \begin{enumerate}
                 \item $\mbox{Arg}Z_{[\phi_i]}(\phi(e)-\epsilon v_i)\leq \mbox{Arg}Z_{[\phi_j]}(\phi(e)-\epsilon v_j)$ for $0<\epsilon \ll 1$, if $i<j$. 
                 \item The above equality holds and $i<j$ implies $|Z_{[\phi_i]}(\phi(e))|\leq Z_{[\phi_j]}(\phi(e))|$. 
              \end{enumerate}
           Assume that $w_{ij},j=1,\cdots,l_i$ are the weights of the edges attached to the pairs $(v_i, [\phi_i])$ and ordered in the way that $w_{ij}\leq w_{ij'}$ if $j<j'$. 
 Then there exists $\tilde{\phi}:\tilde{T}\rightarrow \mathbb{R}^2$ with no contracted edges and weight $\tilde{w}:C_1(\tilde{T})\rightarrow \mathbb{N}$,
   \begin{align*}
      \tilde{w}(e')=\begin{cases}
             w(e'), & e\in T'  \\
             w_i, & e=e_i,
      \end{cases}
   \end{align*} such that the balancing condition (\ref{1099}) is satisfied. Moreover, the tropical curve $(\tilde{\phi},\tilde{T},\tilde{w})$ is in the standard position with respect to $\{v_i,w_{ij}\}$.   
\item Two admissible tropical discs $(\phi,T,w)$ and $(\phi',T',w')$ are equivalent if there exists a homeomorphism $f:T\rightarrow T'$ such that $\phi'\circ f=\phi$ and $w'\circ f=w$.
       \end{enumerate}

    \end{definition}

Now we will define a weighted count for tropical discs on K3 surfaces following \cite{L4}, which is motivated from the work of \cite{GPS}.

 \begin{definition}\label{1011}
   \begin{enumerate}
     \item Let $\phi:T \rightarrow B_{\vartheta}$ be an admissible tropical disc with the stop $u\in B_0$. Then we define its weight of $\phi$ to be 
     \begin{align} \label{4029}
         \mbox{Mult}(\phi):=\prod_{\substack{v\in C^{int}_0(T)\\ v: trivalent}}\mbox{Mult}_v(\phi)\prod_{v\in C^{ext}_0(T)\backslash \{u\}}\frac{(-1)^{w_v-1}}{w_v^2}  \prod_{T_e: \phi(e) \mbox{is a point}}|\mbox{Aut}(\bold{w}_{T(e)})|^{-1},
     \end{align} where the notation is explained below:
     \begin{enumerate}
        \item Here we use the notation in Definition \ref{1012}. Then we set $\bold{w}_{T_e}=(\bold{w}_1,\cdots,\bold{w}_n)$. The last product factor in (\ref{4029}) doesn't repeat the factor if $T_e=T_{e'}$. 
        \item Let $\bold{w}=(\bold{w}_1,\cdots, \bold{w}_n)$ be a set of weight vectors $\bold{w}_i=(w_{i1},\cdots, w_{il_i})$, for $i=1,\cdots, n$. We set 
          \begin{align*}
             a^i_n=|\{w_{ij}|w_{ij}=n\}|
          \end{align*} and
          \begin{align*}
             b_l=\#\{i|Z_{[\phi_i]}(\phi(e))\in l\}
          \end{align*} for any ray $l\in \mathbb{C}^*$. 
          Define
          \begin{align*}
            |\mbox{Aut}(\bold{w})|=\bigg(\prod_{i}\prod_{n\in \mathbb{N} \atop a^i_n\neq 0} (a^i_n)! \bigg) \prod_{l} (b_l)!,
          \end{align*} where the first factor is the size of the subgroup of the permutation group $\prod_i
          \Sigma_{l_i}$ stabilizing $\bold{w}$.
     \end{enumerate}

     \item Let $u\in B_0$ and $\gamma\in H_2(X,L_u)$ such that $u\notin W'^{trop}_{\gamma}$. We define the tropical discs counting invariant $\tilde{\Omega}^{trop}(\gamma;u)$ to be 
        \begin{align*}
           \tilde{\Omega}^{trop}(\gamma;u):=\sum_{\phi} \mbox{Mult}(\phi) ,
        \end{align*} where the sum is over all equivalent classes of admissible tropical discs on $B_{\mbox{Arg}Z_{\gamma}(u)-\pi/2
        }$ with stop at $u$ such that $[\phi]=\gamma$.        
   \end{enumerate}
     \end{definition}

\begin{rmk}\begin{enumerate}
  \item It is proved in \cite{GPS} that the tropical curves counting $N^{trop}_{\{\partial \gamma_i\}}(\mathbf{w})$ (in the proof of Theorem \ref{3011}) is independent of the generic arrangement of the affine lines. 
 The technical and artificial definition of the admissibility is to avoid over count of the tropical discs. 
  \item  The definition of the weighted counting of tropical discs $\tilde{\Omega}^{trop}(\gamma;u)$ does not depend on the choices of the K\"ahler class $[\omega]$ nor the $\mathbb{C}^*$-scaling of the 
  holomorphic volume form $\Omega$.
  \item The central charge can be viewed as the weighted sum of affine length of each edge. 
\end{enumerate}
 
\end{rmk}

The definition of the tropical discs counting comes from the following reason: assume that the affine structure with singularities on $B$ comes from the complex affine structure of a special Lagrangian fibration in K3 surface, the hyperK\"ahler rotation will induce an $S^1$-family of affine structure with singularities, which we will denote them by $\{B_{\vartheta}\}$ and $B_{\vartheta=0}=B$ \cite{L4}.

The following is an analogue of Gromov compactness theorem for tropical discs.
\begin{lem} \label{3006}
   Fix $u\in B_0$ and $\lambda>0$, there exist finitely many relative classes $\gamma\in H_2(X,L_u)$ such that $\tilde{\Omega}^{trop}(\gamma;u)\neq 0$ and $|Z_{\gamma}(u)|<\lambda$. 
\end{lem}
\begin{proof}
   First notice that $|Z_{\gamma}|$ is the sum of affine length of the edges in the image of tropical discs.    Consider the intersections of initial rays from each pair of singularities. Let $\lambda_0$ be the smallest sum of affine length of initial rays to the intersection point. 
   
   Given an admissible tropical disc $(\phi,T,w)$, with $|Z_{[\phi]}|<\lambda$, there exists a unique tropical disc $(\phi',T',w')$ such that $\phi'$ is injective and $\phi$ factors through $\phi'$. In particular, $Z_{[\phi']}=Z_{[\phi]}$. Then we have 
     \begin{align*}
        (|C^{int}_0(T')|+1+|C^{ext}_0(T')-C^{int}_0(T')|)\lambda_0< |Z_{[\phi']}|<\lambda.
     \end{align*} Thus, the number of vertices $|C_0(T')|$ is bounded and there are only finitely many such trees. Let $w'_i$ be the weights on the external wedges of $T'$, then 
       \begin{align*}
         (\sum_i w'_i )\lambda_0< Z_{[\phi']}
       \end{align*} and there are finitely many possible assignment of weights on external edges of a fixed tree. The weight and the direction of the rest of the internal edges are determined by the balancing conditions. 
     
   For a fixed tropical disc $(\phi',T',w')$, there are only finitely many admissible tropical discs factoring through it from the definition of admissibility. Actually, we prove a stronger statement: fix $\lambda>0$. There are only finitely many tropical discs up to equivalent classes, where two tropical discs are equivalent if one can be derived from the other by naturally extending the edge adjacent to the root and the associated relative class are the same (up to parallel transport).
\end{proof} 

\subsection{Wall-Crossing Formula and Properties of Tropical Discs Countings}

Before we talk about the wall-crossing formula, we will introduce the notion of central charge and wall of marginal stability. First there exists a holomorphic function called central charge 
  \begin{align}\label{4010}
      Z: \Gamma=&\cup_{u\in B_0}H_2(X,L_u) \rightarrow \mathbb{C} \notag\\
         &\gamma_u \longmapsto Z_{\gamma_u}(u):=\int_{\gamma_u}\Omega,  \end{align} which is well-defined since $\Omega$ restrict to zero on the fibres. 
 \begin{definition}
    Let $\gamma\in H_2(X,L_{u_0})$, locally we define the locus $W'^{trop}_{\gamma}$ to be 
       \begin{align*}
          W'^{trop}_{\gamma}=\bigcup_{\gamma_1+\gamma_2=\gamma \atop \langle\gamma_1,\gamma_2\rangle\neq 0} W'^{trop}_{\gamma_1,\gamma_2},
       \end{align*} where
       \begin{align*}
          W'^{trop}_{\gamma_1,\gamma_2}=\{u\in B| \mbox{Arg}Z_{\gamma_1}=\mbox{Arg}Z_{\gamma_2} \mbox{ and there exist tropical discs } \\ \mbox{ of relative classes $\gamma_1,\gamma_2$ ending on $u$ such that $\langle \gamma_1,\gamma_2\rangle\neq 0$.}\}.
       \end{align*}
 \end{definition} The locus $W'^{trop}_{\gamma_1,\gamma_2}$ is a locally closed subset of the zero locus of a harmonic function, which is locally union of finitely many smooth curves on $B$. 
\begin{lem} \label{1911}
  Given a tropical disc $(\phi,T,w)$ and $e\in C_1(T)$. If $\phi(e)$ is not a point, then $\phi(e)$ is part of
   the trajectory of a gradient flow
     \begin{align*}
              \frac{d}{dt}\phi(t)=\nabla |Z_{\gamma}(\phi(t))|^2,
          \end{align*} for some relative class $\gamma$, where $\nabla$ is respect to a K\"ahler metric on $B$. 
\end{lem}
 \begin{proof}
  Since $\phi(e)$ is an affine segment. Locally, it is given by the equation $f_{\gamma_{\bar{e}}}=c$ for some relative class $\gamma_{\bar{e}}\in H_2(X,L)$ such that $\mbox{Arg}Z_{\gamma_{\bar{e}}}$ is constant along $\bar{e}$ by Definition \ref{303}. On the other hand, for a given relative class $\gamma\in H_2(X,L)$
    the trajectories of the equation 
       \begin{align*}
           \frac{d}{dt}\phi(t)=\nabla |Z_{\gamma}(\phi(t))|^2
       \end{align*} are also characterized by the same property. Indeed, we have 
         \begin{align*}
           \frac{d}{dt}\mbox{Arg}Z_{\gamma}(\phi(t))=(\nabla F_{\gamma})\mbox{Arg}Z_{\gamma}=J\nabla F_{\gamma}(\log{|Z_{\gamma}}|)=0,
         \end{align*} where we denote $F_{\gamma}=|Z_{\gamma}|^2$ locally. The second equality follows from the Cauchy-Riemann equation of $Z_{\gamma}$. The third equation holds because for any function $f$, its gradient $\nabla f$ is perpendicular to its level set. So $J\nabla f(f)=0$. 
          In particular, the tropical disc $(\phi,T,w)$ can vary smoothly with respect to $\vartheta$ if $\phi(C^{int}_0(T))\cap \Delta =\emptyset$. 
 \end{proof}
\begin{rmk}
 Since $Z_{\gamma}$ is a holomorphic function, the gradient flow lines of phase $\vartheta_1$ will be on the right of the gradient flow line of phase $\vartheta_2$ if $\vartheta_1<\vartheta_2$. 
\end{rmk}

For a fixed fibre $L_{u}, u\in B_0$ and $\gamma \in H_2(X,L_{u})$ being primitive, We will associate a generating function of the tropical disc of multiple covers of $\gamma$ given by 
     \begin{align} \label{78}
        \log f^{trop}_{\gamma}(u)= \sum_{d=1}^{\infty}  d\tilde{\Omega}^{trop}(d\gamma;u)\big(T^{Z_{\gamma}(u)}z^{\gamma}\big)^d \in \mathbb{C}[[z^{\gamma}]]\otimes \Lambda_+^{\mathcal{S}_{\gamma}(u)},
     \end{align}  where $\mathcal{S}_{\gamma}(u):=\mathbb{R}_{>0}Z_{\gamma}(u)$ is a ray. 
The function $f^{trop}_{\gamma}$ is known as the slab function in \cite{GS1}. In particular, we have 
   \begin{align*}
      &f^{trop}_{\gamma}(u) \equiv 1 \hspace{4mm}( \mbox{mod }\Lambda_+^{\mathcal{S}_{\gamma}(u)})\\
      &f^{trop}_{\gamma}(u) \equiv 1 \hspace{4mm}( \mbox{mod } (z^{\partial \gamma})).
   \end{align*} 

 Then given $u\in B_0$, $\lambda>0$ and a sector $\mathcal{S}$, we associate a symplectomorphism of the "quantum torus $(\mathbb{C}^*)^2$" as follows: first we let $\tilde{\theta}^{trop}_{\gamma,<\lambda}(u)$ be an automorphism of algebras through the inclusion $\Lambda^{\mathcal{S}_{\gamma}}/F^{\lambda}\Lambda^{\mathcal{S}_{\gamma}}\subseteq \Lambda^{\mathcal{S}}/F^{\lambda}\Lambda^{\mathcal{S}}$ and $\Lambda[[z^{ \gamma}]]\subseteq \Lambda[[H_2(X,L_u)]]$,
   \begin{align} \label{3002}
        \tilde{\theta}^{trop}_{\gamma,<\lambda}(u):(\Lambda^{\mathcal{S}}/F^{\lambda}\Lambda^{\mathcal{S}})[[H_2(X,L_u)]] &\rightarrow (\Lambda^{\mathcal{S}}/ F^{\lambda}\Lambda^{\mathcal{S}})[[H_2(X,L_u)]] \notag \\
           z^{\gamma'} & \longmapsto  z^{ \gamma'} (f^{trop}_{\gamma}(u))^{\langle \gamma', \gamma \rangle}
    \end{align} and let $\tilde{\theta}_{\gamma}(u):=\lim \tilde{\theta}^{trop}_{\gamma}(u)$. Then we define
   \begin{align} \label{3007}
     \Theta_{\mathcal{S},<\lambda}^{trop}(u):=\prod^{\curvearrowright}_{\gamma\in H_2(X,L_u) \mbox{ primitive}: \atop \mbox{Arg}Z_{\gamma}(u)\in \mathcal{S} } \tilde{\theta}^{trop}_{\gamma,<\lambda}(u).
  \end{align} Here the product $\overset{\curvearrowright}{\prod}$ is taken in the order of $\mbox{Arg}Z_{\gamma}(u)$ and by Lemma \ref{3006} that (\ref{3007}) is a finite product. The pairing $\langle \gamma',\gamma\rangle$ is the natural pairing of the boundaries in $H_1(L_u)$. Then the limit
     \begin{align*}
         \Theta^{trop}_{\mathcal{S}}(u):=\lim_{\lambda\rightarrow \infty} \Theta_{\mathcal{S},<\lambda}^{trop}(u)\in \mbox{Aut}(\lim_{\leftarrow}\Lambda^{\mathcal{S}}/(F^{\lambda}\Lambda^{\mathcal{S}})[[H_1(L_u)]])
     \end{align*} exists via Theorem \ref{3011} and the uniqueness of scattering diagram.

  \begin{definition}
  We will call a transformation of the form in (\ref{3002}) an elementary transformation and call $f^{trop}_{\gamma}$ the associated slab function.
  \end{definition}
  It is straight-forward to check that if $\langle \gamma_1,\gamma_2\rangle=0$, then $[\tilde{\theta}^{trop}_{\gamma_1},\tilde{\theta}^{trop}_{\gamma_2}]=0$.
 The symplectomorphism $\Theta(u)$ is well-defined in $\Lambda_+$-adic topology due to the above algebra calculation 
 and Lemma \ref{3006}.

Given a path $\phi:[0,1] \rightarrow B_0$ such that $\phi(0)=u_0$ and $\phi(1)=u_1$, the parallel transport $T_{u_0,u_1}:H_2(X,L_{u_0})\rightarrow H_2(X,L_{u_1})$ naturally induces a well-defined transformation (which we will still denote by $T_{u_0,u_1}$ and also the induced transformation on the quotients) 
   \begin{align*}
      \Lambda[[H_2&(X,L_{u_0})]] \rightarrow \Lambda[[H_2(X,L_{u_1})]] \\
          &z^{ \gamma_{u_0}} \mapsto T^{Z_{\gamma}(u_1)-Z_{\gamma}(u_0)}z^{\gamma_{u_1}}.
   \end{align*}
 We will also use the same notation for the induced action on $\Theta^{trop}_{\mathcal{S}}(u)$. 
 The tropical disc counting invariants $\tilde{\Omega}^{trop}(\gamma;u)$ satisfy the Kontsevich-Soibelman wall-crossing formula in the following sense:

\begin{thm}\label{3008} Assume that $X$ satisfies the generic condition (*) (see Definition \ref{9002}). 
 Given $u\in B_0$, $\lambda>0$ and a sector $\mathcal{S}$. Assume that there exists no $\gamma\in H_2(X,L_u)$ such that $\tilde{\Omega}'(\gamma;u)\neq 0$, $|Z_{\gamma}(u)|<\lambda$ and $\mbox{Arg}Z_{\gamma}(u)\in \partial \mathcal{S}$. Then there exists a neighborhood $U_{\lambda}\ni u$ and $\lambda'=\lambda'(\lambda)>0$ such that 
     \begin{align*}
        \Theta^{trop}_{\mathcal{S}}(u_2)=T_{u_1,u_2}\Theta^{trop}_{\mathcal{S}}(u_1) \hspace{4mm}(\mbox{mod }F^{\lambda'}\Lambda^{\mathcal{S}}),
     \end{align*} for generic $u_1,u_2\in U$ . Moreover, we have $\lim_{\lambda\rightarrow \infty}\lambda'(\lambda)=\infty$.  
     In other words, the symplectomorphism $\Theta^{trop}_{\mathcal{S}}(u)$ is "invariant" under the identification of parallel transport. 
\end{thm}

  First we have the following corollary from \cite{GPS}. 
\begin{thm}\cite{GPS}\cite{L4} \label{3011} 
  Assume $X$ satisfies the generic condition (*).
Fix a generic $\vartheta\in S^1$. 
   Let $u\in B_0$ and $\lambda>0$. Let $\{l_{\gamma_i}\}$ the set of  affine line (and rays) with respect to the affine manifold $B_{\vartheta}$ labeled by $\gamma_i$ passing through $u$ such that $|Z_{\gamma_i	}(u)|<\lambda$. Let $\phi$ be an oriented circle around $u$ small enough such that it intersects each of $l_{\gamma_i}$ exactly twice (and once) and transversally at $p_i$. Then for any $\lambda$,
      \begin{align*}
         \prod_i \bigg(T_{p_i,u}\tilde{\theta}_{\gamma_i}(p_i)\bigg)^{\epsilon_i}=\mbox{Id} \hspace{4mm} ( \mbox{ mod }F^{\lambda}\Lambda^{\mathbb{R}_{>0}e^{i\vartheta}}),
      \end{align*} where the order of the product is with respect to the order of the intersections of $l_{\gamma_i}$ with $\phi$. The exponent $\epsilon_i$ is the sign of the pairing between the tangent of the circle and the tangent of the affine line $l_{\gamma_i}$ (with the direction $|Z_{\gamma_i}|$ is decreasing).

\end{thm}
\begin{proof} The generic choice of $\vartheta$ implies that all the affine lines/rays go from one side of the wall $W_{\gamma}^{'trop}\ni u$.
 Let $u_+,u_-$ be on different side of $W'^{trop}_{\gamma}$ with the following properties: 
 \begin{enumerate}
    \item $\mbox{Arg}Z_{\gamma}(u_+)=\mbox{Arg}Z_{\gamma}(u_-)$ and $|Z_{\gamma}(u_+)|>|Z_{\gamma}(u_-)|$.  
    \item there exists no $u'$ in the affine line between $u_+,u_-$ (in $B_{\vartheta},\vartheta=\mbox{Arg}Z_{\gamma}(u_+)$) with $\gamma'\in H_2(X,L_u)$ such that $u\in W'^{trop}_{\gamma'}$ and $|Z_{\gamma'}(u)|<\lambda$. 

 \end{enumerate}
 For each $\lambda>0$, the above conditions can be achieved by choosing $u_+,u_-$ close enough to $W'^{trop}_{\gamma}$. Assume that $u$ is the intersection point of the affine line segment between $u_+,u_-$ and $W'^{trop}_{\gamma}$. Let $\gamma_i\in H_2(X,L_u)$, $i=1,\cdots, n$, be the primitive relative classes which can be realized as tropical discs of the same phase and $|Z_{\gamma_i}(u)|<\lambda$. Then the additional tropical discs with stop at $u_+$ of relative class $d\gamma$ are given by those tropical discs with a subtree $T_e$ maps to $u$ and the complement of $T_e$ and edge adjacent to $x$ are tropical sub-discs $\phi_j=(\phi|_{T_j},T_j,w|_{T_j})$. Assume that $[\phi_j]/\gamma_i$ (if not divisible then we set it to be zero) takes non-zero values $ w_{i1}\leq \cdots \leq w_{il_i}$ and $\sum_i |\bold{w}_i|\gamma_i=d\gamma$. For fixed tropical sub-disc $\phi_i$, one can glue to a tropical curve $(\tilde{\phi},\tilde{T},\tilde{w})$ (see Definition \ref{1012}) to a tropical disc with stop at $u_+$ and relative class $d\gamma$. Each tropical disc has its weight given by
    \begin{align} \label{2017}
        \frac{\mbox{Mult}(\tilde{\phi})}{|Aut(\mathbf{w})|}\prod_i \mbox{Mult}(\phi_i).
    \end{align}
  Recall that $N^{trop}_{\{\partial \gamma_i\}}(\mathbf{w})=\sum_{\tilde{\phi}}\mbox{Mult}(\tilde{\phi})$ is the counting of tropical rational curves defined in \cite{GPS}. Then the difference of the tropical discs counting invariants is given by summing (\ref{2017}) over all possible admissible $\tilde{T}$, $\phi_i$, $w_{ij}$. By induction on the number of vertices, we have  
    \begin{align*}
     \tilde{\Omega}^{trop}(d\gamma;u_+)-\tilde{\Omega}^{trop}(d\gamma;u_- )= \sum_{\mathbf{w}:\sum|\mathbf{w_i}|\gamma_i=d\gamma}\frac{N^{trop}_{\{\partial\gamma_i\}}(\mathbf{w})}{|Aut(\mathbf{w})|} \bigg(\prod_{1 \leq i \leq n, 1 \leq j \leq l_i} \tilde{\Omega}^{trop}(w_{ij}\gamma_i)
               \bigg),
    \end{align*}  and the theorem follows from Theorem 2.8 \cite{GPS}. 
\end{proof}

\begin{thm} \cite{L4} \label{3037}
  Assume that $u\notin W'^{trop}_{\gamma}$, then there exist a neighborhood $\mathcal{U}\ni u$ such that tropical discs counting is locally a constant, namely
   \begin{align*}
  \tilde{\Omega}^{trop}(\gamma,u)=\tilde{\Omega}^{trop}(\gamma;u'),
  \end{align*}
  for every $u'\in \mathcal{U}$. 
\end{thm} 
\begin{proof}
   First we assume that for all admissible tropical discs $(\phi,T,w)$ stop at $u$ with $[\phi]=\gamma$ satisfy that $\phi(C^{int}_0(T))\subseteq B_0$. Then the theorem follows from Lemma \ref{1911} and the fact that the solution of the first order ordinary differential equations has smooth dependence of the initial values. It worth noticing that tropical discs with vertices of valency larger than $3$ might not deform smoothly respect to $\vartheta$. 
      
   Assume that there exists a tropical disc $(\phi,T,w)$ with respect to $\vartheta$, $[\phi]=\gamma$, and pass through a singularity $p$.   Fix $\lambda>0$, 
  we may assume that the affine line in phase $\vartheta$ from $p$ to $u$ does not intersect any other $W'_{\gamma'}$ with $|Z_{\gamma}|<\lambda$. Thus, $\tilde{\Omega}(\gamma;u)=\tilde{\Omega}(\gamma;u')$ for $u'$ on the affine segment. Choose $u_+,u_-$ near $u$ such that $\mbox{Arg}Z_{\gamma}(u_+)>\vartheta$ and $\mbox{Arg}Z_{\gamma}(u_-)<\vartheta$. Let $\gamma_e$ be the relative class of Lefschetz thimble associated to $p$. Then locally $l_{\pm\gamma_e}$ divide a neighborhood of $p$ in to two regions.
  
  Let $l_{\gamma_i}$ be the affine lines of phase $\vartheta$, labeled by $\gamma_i$, passing through $p$ and $|Z_{\gamma_i}(p)|<|Z_{\gamma}(p)|$. Choose $0<\epsilon \ll 1$, and denotes $\{l^{\pm}_i\}$ be the sets of affine lines with respect to $\vartheta\pm \epsilon$ corresponding to the deformation of $\{l_{\gamma}\}$ and those affine rays with respect to $\vartheta\pm \epsilon$ labeled with relative classes that admits tropical discs containing some $l_{\gamma_i}^{\pm}$. 
  
   Choose $u'_{\pm}$ on $l_{\pm\gamma_e}$ close to $p$ such that $-Z_{\gamma_e}(u'_+)=Z_{-\gamma_e}(u'_-)$. Choose counter-clockwise loops $\phi_{\pm}$ around $u'_{\pm}$ such that they intersect all the affine lines (and rays) in $\{l_{\gamma_i}^{\pm}\}$ transversally and exactly twice (and once respectively). We decompose $\phi_{\pm}$ into four arcs $\phi^{\pm}_i$, $i=1,2,3,4$ such that $\phi_1^{\pm},\phi_3^{\pm}$ only intersect $l_{\pm\gamma_e}$, $\phi^{\pm}_2$ is contained in the region with $u$ and $\phi^{\pm}_4$ is contained in the region without $u$ (See Figure \ref{fig:99}). Let $F_{\phi^{\pm}_i}(u'_{\pm})$ to be the composition of the transformation (parallel transport to $u'_{\pm}$) attached to $l^{\pm}_{\gamma_i}$ such that $\phi^{\pm}_i$ intersect and the composition order is with respect to the order of intersection of $l^{\pm}_{\gamma_i}$ with $\phi^{\pm}_i$. Then by Theorem \ref{3011}, we have 
        \begin{align*}
          F_{\phi^+_4}(u'_+)\circ F_{\phi^+_3}(u'_+) \circ F_{\phi^+_2}(u'_+)\circ F_{\phi^+_1}(u'_+)  &= \mbox{Id}  \\
          F_{\phi^-_4}(u'_-)\circ F_{\phi^-_3}(u'_-) \circ F_{\phi^-_2}(u'_-)\circ F_{\phi^-_1}(u'_-)  &= \mbox{Id}  
        \end{align*}
     \begin{figure}
                         \begin{center}
                         \includegraphics[height=3in,width=6in]{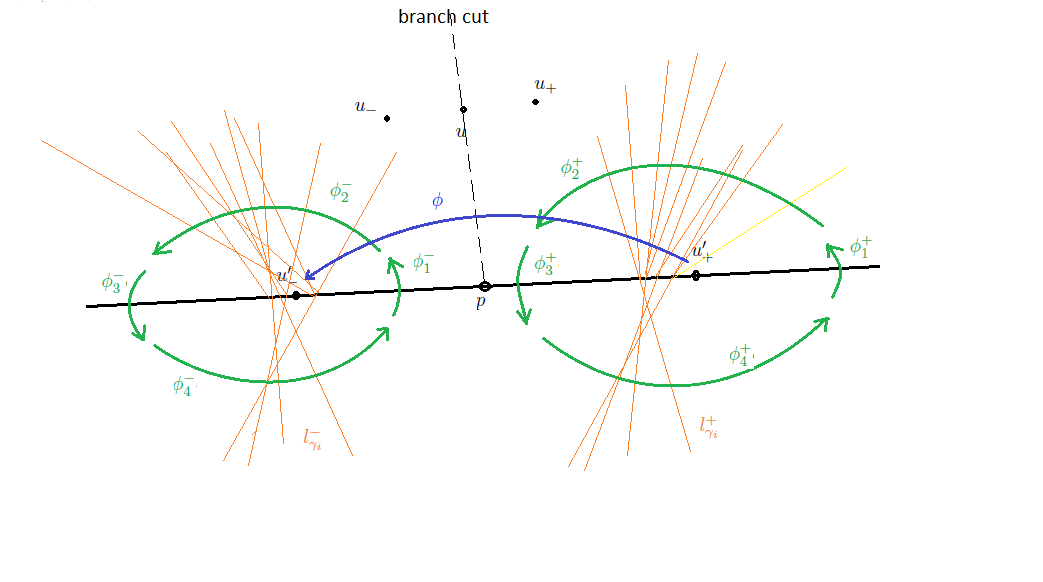}
                         \caption{When tropical discs vary across a singularity.}
                          \label{fig:99}
                         \end{center}
                         \end{figure}
 Let $\alpha$ denotes the action of (counter-clockwise) monodromy around $p$, 
   \begin{align*}
       \alpha(z^{\gamma'})=z^{\gamma'-\langle\gamma_e,\gamma'\rangle\gamma_e}.
   \end{align*}  
 For $u'$ near $p$, we also define the transformation 
    \begin{align*}
       \theta_{\pm\gamma_e}(u')(z^{\gamma'})= z^{\gamma'} (1+z^{\pm \gamma_e}T^{\pm Z_{\gamma_e}(u')})^{\pm 1}.
    \end{align*}  
   
Notice that for any $\gamma'$, we have
     \begin{align*}
         \alpha^{-1}\circ 
         (\theta_{-\gamma_e}(u'_+))^{-1}(z^{\gamma'})&= \alpha^{-1}\big(z^{\gamma'}(1+T^{-Z_{\gamma_e}(u'_+)}z^{-\gamma_e})^{\langle \gamma_e,\gamma'\rangle}\big) \\
         &=z^{\gamma'+\langle \gamma_e,\gamma\rangle \gamma_e} (1+T^{-Z_{\gamma_e}(u'_+)}z^{-\gamma_e})^{\langle\gamma_e,\gamma'\rangle}\\
         &=z^{\gamma'}(1+T^{Z_{\gamma_e}(u'_+)}z^{ \gamma_e})^{\langle \gamma_e,\gamma\rangle}T^{-Z_{\gamma_e}(u'_+)\langle \gamma_e,\gamma\rangle}.
      \end{align*}  
Therefore, we have 
  \begin{align}\label{4999}
     \alpha^{-1}\circ 
     (\theta_{-\gamma_e}(u'_+))^{-1}=T^{-Z_{\gamma_e}(u'_+)\langle\gamma_e,\gamma\rangle}\theta_{\gamma_e}(u_-).
  \end{align}  
     By definition, we have $F_{\phi^{\pm}_3}(u'_{\pm})={\theta_{\pm\gamma_e}(u'_{\pm})}^{\mp 1}$ and $F_{\phi^{\pm}_1}(u'_{\pm})=\theta_{\pm\gamma_e}(u'_{\pm})^{\pm 1}$. 
 From (\ref{4999}), we have 
   \begin{align}\label{4996}
      \mbox{Id}&= F_{\phi^+_4}(u'_+)\circ {\theta}_{\gamma_e}(u'_+)^{-1}\circ F_{\phi^+_2}(u'_+)\circ {\theta}_{\gamma_e}(u'_+)  \notag \\
      &= F_{\phi^+_4}(u'_+)\circ {\theta}_{-\gamma_e}(u'_+)^{-1}\circ ( \alpha \circ F_{\phi^+_2}(u'_+)\circ \alpha^{-1} ) \circ \theta_{\gamma_e}(u'_+)^{-1} \\
   \end{align}
       Apply $T_{u'_+,u'_-}$ (along a path contained in the region does not contain $u$) to both sides of (\ref{4996}), we have $F_{\phi^-_2}(u'_-)=T_{u'_+,u'_-}\alpha\circ F_{\phi^+_2}(u_+)\circ \alpha^{-1}$ and we have 
        \begin{align*}
           \tilde{\Omega}(\gamma;u_+)=\tilde{\Omega}(\gamma;u_-). 
        \end{align*} 
  Let $\tilde{\Omega}(\gamma;u_+)_1$ to be the contribution from tropical discs with images which contain $p$ and locally intersect both regions, $\tilde{\Omega}(\gamma;u_+)_2$ to be the contribution from tropical discs with images which do not contain $p$ and locally intersect both regions, $\tilde{\Omega}(\gamma;u_+)_3$ to be the contribution from tropical discs with images locally only intersect the region contains $u$. Then 
    \begin{align*}
       \tilde{\Omega}(\gamma;u_+)=\tilde{\Omega}(\gamma;u_+)_1+\tilde{\Omega}(\gamma;u_+)_2+\tilde{\Omega}(\gamma;u_+)_3.
    \end{align*} Similarly, we define $\tilde{\Omega}(\gamma;u_-)_1,\tilde{\Omega}(\gamma;u_-)_2$. Then 
    we have 
      \begin{align*}
         \tilde{\Omega}(\gamma;u_+)_1=\tilde{\Omega}(\gamma;u_-)_2 \\
          \tilde{\Omega}(\gamma;u_+)_2=\tilde{\Omega}(\gamma;u_-)_1 \\
          \tilde{\Omega}(\gamma;u_+)_3=\tilde{\Omega}(\gamma;u_-)_3.
      \end{align*}
    In particular, this implies that
      \begin{align*}
        \tilde{\Omega}(\gamma;u)=\tilde{\Omega}(\gamma;u_+)_1+\tilde{\Omega}(\gamma;u_-)_1+\tilde{\Omega}(\gamma;u_+)_3=\tilde{\Omega}(\gamma;u_+)     
      \end{align*} and this finishes the proof of the theorem.
\end{proof}

Given $1$-parameter family of central charges $Z^t:\Gamma \rightarrow \mathbb{C}$, $t\in [0,1]$ such that $Z^0=Z$, then each central charge will induce an $S^1$-family of affine structures on $B_0$ and the corresponding tropical disc counting invariant which we denote them by $\tilde{\Omega}_t^{trop}(\gamma;u)$. 
\begin{thm}  \label{9011}
Given the pair $(u,\gamma)\notin W'^{trop}(\gamma;u)$, then $\tilde{\Omega}^{trop}_t(\gamma;u)$ will be well-defined for $t\ll 1$ and 
\begin{align*}
  \tilde{\Omega}^{trop}(\gamma;u)=\tilde{\Omega}_t^{trop}(\gamma;u),
\end{align*} for $t\ll 1$. 
\end{thm}
\begin{proof}
If $(u,\gamma)\notin W'^{trop}_{\gamma}$, then it avoids the corresponding walls defined by $Z^t$ for $t\ll 1$. 
The proof of Theorem \ref{9011} is similar to Theorem \ref{3037} so we will skip it. 
\end{proof}

\begin{proof} (of Theorem \ref{3008})
   First we assume that $u\notin W'^{trop}_{\gamma}$ for any $\gamma\in H_2(X,L_u)$ with $|Z_{\gamma}(u)|<\lambda$ and $Z_{\gamma}(u)\in \mathcal{S}$. Together with the assumption of the Theorem \ref{3008} and Lemma \ref{3006}, there exists a neighborhood $U\ni u$ such that for every $u'\in U$, we have 
   \begin{enumerate}
      \item there exits no $\gamma\in H_2(X,L_{u'})$ with $|Z_{\gamma}(u')|<\lambda$, $\tilde{\Omega}^{trop}(\gamma;u')\neq 0$ and $\mbox{Arg}Z_{\gamma}(u' )\in \partial \mathcal{S}$. 
      \item there exists no $\gamma\in H_2(X,L_{u'})$ with $|Z_{\gamma}(u')|<\lambda$, $Z_{\gamma}(u')\in \mathcal{S}$ and $u'\in W'^{trop}_{\gamma}$. 
   \end{enumerate}
   Then the relative class $\gamma$ with $|Z_{\gamma}(u')|<\lambda$, $Z_{\gamma}(u')\in \mathcal{S}$ and $\tilde{\Omega}^{trop}(\gamma;u')\neq 0$ for every $u'\in U$ can be identified via the parallel transport. From Theorem \ref{3037}, we have 
     \begin{align} \label{2998}
         T_{u_1,u_2}f^{trop}_{\gamma}(u_1)=f^{trop}_{\gamma}(u_2) \hspace{4mm}(\mbox{mod }T^{\lambda})
     \end{align} for all $u_1,u_2\in U$ and $\gamma$ such that $|Z_{\gamma}|<\lambda$ on $U$. Moreover, the order of $\mbox{Arg}Z_{\gamma}(u')$ with $\gamma\in H_2(X,L_{u'})$, $\tilde{\Omega}(\gamma;u')\neq 0$ is the same for every $u'\in U$. Then the theorem follows from Lemma \ref{2999}. 
   
   Now assume that $u\in W'^{trop}_{\gamma_1}\cap \cdots \cap W'^{trop}_{\gamma_k}$ for one $\gamma_i\in H_2(X,L_u)$ with $|Z_{\gamma}(u)|<\lambda$, $\mbox{Arg}Z_{\gamma_i}(u)\in \mathcal{S}$. We take $U\ni u$ be a neighborhood such that for any $u'\in U$, we have 
     \begin{enumerate}
        \item there is no $\gamma\in H_2(X,L_{u'})$ with $|Z_{\gamma}(u')|<\lambda$, $\tilde{\Omega}^{trop}(\gamma;u')\neq 0$ and $\mbox{Arg}Z_{\gamma}(u')\in \partial\mathcal{S}$.
         \item there is no $\gamma\in H_2(X,L_{u'})$ with $|Z_{\gamma}(u')|=\lambda$, $\tilde{\Omega}^{trop}(\gamma;u')\neq 0$ and $\mbox{Arg}Z_{\gamma}(u')\in \mathcal{S}$. (this can be achieved by slightly enlarge $\lambda$)
        \item there is no $\gamma\in H_2(X,L_{u'})$ with $|Z_{\gamma}(u')|<\lambda$, $u'\in W'^{trop}_{\gamma}\cap U \neq W'^{trop}_{\gamma_i}\cap U$.
     \end{enumerate}
      Given any $u_1,u_2\in U$, the relative class $\gamma$ with $\tilde{\Omega}(\gamma;u_i)\neq 0$, $\mbox{Arg}Z_{\gamma}(u_i)\in \mathcal{S}$ and $|Z_{\gamma}(u_i)|<\lambda$ are identified except those annihilate/creation from the walls $W'^{trop}_{\gamma_i}$. 
      It suffices to consider the situation when  $u_1,u_2$ are closed to a single wall $W'^{trop}_{\gamma_i}$ and on the different sides (otherwise, it reduces to the previous situation), for some $i$.
      Choose a path $\phi$ connecting $u_1,u_2$ passing through $W'^{trop}_{\gamma_i}$ exactly once at $p$ and do not intersect $W'^{trop}_{\gamma_j}$ for $j\neq i$. Let $\{l_{\gamma'_i}\}$ be the sets of affine lines and rays such that $\gamma'_i$ satisfies $\tilde{\Omega}^{trop}(\gamma'_i;p)\neq 0$ and $\mbox{Arg}Z_{\gamma'_i}(u)=\mbox{Arg}Z_{\gamma_i}(p)$. Let $\phi$ be a small loop around $p$ and falls in $U$. Then $\phi$ is separated by $W'^{trop}_{\gamma_i}$ into two parts $\phi_1,\phi_2$. Assume that $\phi_1$ is in the same sides of $W'^{trop}_{\gamma_i}$ as $u_1$. Then there exists a subsector $\mathcal{S}_i$ contains $Z_{\gamma_i}(p)$ in its interior such that  
      \begin{align} \label{2996}
         T_{u_1,p}\big(\prod_{\gamma:\mbox{Arg}Z_{\gamma}\in \mathcal{S}'}\tilde{\theta}^{trop}_{\gamma_i}(u_1)\big)&=\prod_{\gamma_i:l_{\gamma_i}\cap \phi_1\neq \emptyset} \tilde{\theta}^{trop}_{\gamma_i}(p) \\&=\prod_{\gamma_i:l_{\gamma_i}\cap \phi_2\neq \emptyset} \tilde{\theta}^{trop}_{\gamma_i}(p)= T_{u_2,p}\big(\prod_{\gamma:\mbox{Arg}Z_{\gamma}\in \mathcal{S}'}\tilde{\theta}^{trop}_{\gamma_i}(u_2)\big) \hspace{4mm}(\mbox{ mod } T^{\lambda'})\notag.
      \end{align} Here the first and third equality comes from the choice of $U$ and the second equality comes from Theorem \ref{3011}. Applying $T_{u,u_2}$ on both sides of (\ref{2996}), we have 
        \begin{align*}
            \Theta^{trop}_{\mathcal{S}_i}(u_2)=T_{u_1,u_2}\Theta^{trop}_{\mathcal{S}_i}(u_1) \hspace{4mm}(\mbox{mod } T^{\lambda'}).
        \end{align*} For relative classes $\gamma'$ such that $\tilde{\Omega}^{trop}(\gamma')\neq 0$ and $\mbox{Arg}Z_{\gamma'}\in \mathcal{S}\backslash \mathcal{S}_i$ along $\phi$, they are identified via parallel transport and thus (\ref{2998}) holds for those $\gamma'$. The theorem then follows from Lemma \ref{2999}.

\end{proof}

\begin{rmk}
   The Theorem \ref{3008} is also true under the projection $H_2(X,L_u)\rightarrow H_1(L_u)$. 
\end{rmk}

\begin{definition} 
   The wall of marginal stability for tropical discs are defined by
   \begin{align*}
      W^{trop}_{\gamma}:=\{u\in W'^{trop}_{\gamma}| \tilde{\Omega}^{trop}(\gamma) \mbox{ differs on different sides of $W'^{trop}_{\gamma}$}\}.
   \end{align*} For $u\in W'^{trop}_{\gamma}\backslash W^{trop}_{\gamma}$, we will define $\tilde{\Omega}^{trop}(\gamma;u)$ by the natural continuous extension. 
\end{definition}
\begin{rmk}
   Similar proofs shows that Theorem \ref{3037} holds when $u\notin W^{trop}_{\gamma}$. 
\end{rmk}

\begin{definition}
  A quadratic refinement is an assignment $c:\Gamma \rightarrow \{\pm 1 \}$ such that 
     \begin{align*}
        c(\gamma_1+\gamma_2)=(-1)^{\langle\partial \gamma_1,\partial\gamma_2 \rangle} c(\gamma_1)c(\gamma_2), 
     \end{align*} for any $\gamma_1,\gamma_2\in \Gamma$.
\end{definition}
We will put proof the existence of quadratic refinement in Appendix \ref{2997}. With the notion of quadratic refinement, the slab function $f_{\gamma}$ can be decomposed into unique multiplicative sequence in a unique way as follows:
   \begin{align} \label{77}
       f^{trop}_{\gamma}(u)=\prod_{d\in \mathbb{N}} (1-c(d\gamma)z^{d\partial \gamma})^{d\Omega^{trop}(d\gamma;u)}
   \end{align} with $\Omega^{trop}(\gamma;u)\in \mathbb{Q}$ a priori. In other words, this introduce a multiplicative decomposition of $\tilde{\theta}_{\gamma}(u)$ given by
     \begin{align*}
        \tilde{\theta}^{trop}_{\gamma}(u)=\prod_d(\theta^{trop}_{\gamma}(u))^{\Omega(\gamma;u)}.
     \end{align*}
     Properties of $\tilde{\Omega}^{trop}(\gamma;u)$ will translate to analogue properties of $\Omega^{trop}(\gamma;u)$. The locally constant $\Omega^{trop}(\gamma;u)$ are known as the generalized Donaldson-Thomas invariants. 
Then the integrality conjecture \cite{KS2} states that
  \begin{conj}
     Under above notation and assumption, 
     \begin{enumerate}
        \item $\Omega^{trop}(\gamma;u)\in \mathbb{Z}$ for $u\notin W'^{trop}_{\gamma}$.
        \item ${\Omega}^{trop}(d\gamma;u)=0$ for sufficiently large $d$. 
     \end{enumerate}
     
  \end{conj}
The set of invariants $\{\tilde{{\Omega}}^{trop}(\gamma;u)\}$ can be derived from the set of invariants $\{{\Omega}^{trop}(\gamma;u)\}$ recursively and vice versa. Indeed, we have 
  \begin{align} \label{79}
     \tilde{\Omega}^{trop}(d\gamma)=-\sum_{k|d}c(\frac{d}{k}\gamma)^d\frac{\Omega^{trop}(\frac{d}{k}\gamma)}{k^2}, 
  \end{align} for any $d\in \mathbb{Z}$ from equation (\ref{78}) (\ref{77}). The equation (\ref{79}) is also referred as the multiple cover formula for holomorphic discs \cite{FOOO}.

 \begin{rmk} \label{348}
    Another way to understand the Kontsevich-Soibelman transformation is to consider the Lie algebra structure on $\Lambda^{\mathcal{S}}/F^{\lambda}\Lambda^{\mathcal{S}}[[H_1(L_u)]]$ given by 
       \begin{align*}
          [z^{\partial\gamma_1},z^{\partial\gamma_2}]=\langle\gamma_1,\gamma_2\rangle z^{\partial\gamma_1+\partial\gamma_2}.
       \end{align*}
    Then the transformation $\theta_{\gamma}$ admits another expression given by 
      \begin{align} \label{336}
         \theta_{\gamma}^{-1}(u)=\exp{\bigg(\frac{1}{d}ad\big(Li_2(c(\gamma)T^{Z_{\gamma}(u)}z^{\partial\gamma}))\big)\bigg)},
      \end{align} where $Li_2$ is the dilogarithm function 
       $
          Li_2(z):=\sum^{\infty}_{k=1}\frac{z^k}{k^2}
       $ and $d$ is the divisibility of $\partial\gamma$. More generally, an elementary transformation 
         \begin{align*}
            z^{\partial\gamma'} \mapsto z^{\partial \gamma'}f^{\langle\gamma',\gamma\rangle}
         \end{align*} for $f\in 1+ \Lambda[[z^{\partial\gamma}]]$ can be expressed as product of factors in the form (\ref{336}) or the form $\exp{\tilde{f}}$, where $\tilde{f}\in z^{\partial\gamma}\Lambda [[z^{\partial \gamma}]]$. It is straight-forward to check that $f$ and $\tilde{f}$ determines each other. There exists a natural $q$-deformation of the natural dilogarithm which leads to a natural $q$-deformation of tropical discs counting satisfying the $q$-deformed Kontsevich-Soibelman wall-crossing formula. We will leave the discussion in Appendix \ref{4099}.
 \end{rmk}

\section{Lagrangian Floer Theory of Special Lagrangians in hyperK\"ahler Manifolds} \label{4006}

This section we will review some results in Lagrangian Floer theory ant its application to special Lagrangians in hyperK\"ahler manifolds.
In particular, we will use $\Lambda$ to be the standard Novikov ring
   \begin{align*}
      \Lambda:=\{\sum_{i=0}^{\infty} a_i T^{\lambda_i}| a_i\in \mathbb{C}, \lambda_i \in \mathbb{R}, \mbox{ and $\lambda_i$ increasing}, \lim \lambda_i=+\infty \}.
   \end{align*} and $\Lambda_+$ be the subset of $\Lambda$ with positive valuation. 
 
 Let $(X,\omega)$ be a symplectic manifold with a compatible almost complex structure $J$. Let $L$ be a Lagrangian submanifold in $X$. From the real codimension one boundary of the moduli space of holomorphic discs, Fukaya  proved
 \begin{thm}  \cite{F1}\label{41}
    There exists an gapped filtered $A_{\infty}$ structure $\{m_{k,\gamma}\}_{k\geq 0,\gamma\in H_2(X,L)}$ on the de Rham complex $\Omega^*(L,\Lambda)$ unique up to pseudo-isotopies. 
 \end{thm}     
 From Theorem \ref{41} and Theorem \ref{25}, we have the following corollary:
 \begin{cor} \label{22} \cite{F1}
   There exists an gapped filtered $A_{\infty}$ structure on $H^*(L,\Lambda)$ unique up to pseudo-isotopies. 
 \end{cor}
 A priori, the $A_{\infty}$ structure in Corollary \ref{22} may depend on the choices of the almost complex structures. For any two compatible almost complex structures $J,J'$, one follows Theorem \ref{41} and constructs two $A_{\infty}$ structures $\{m_{k,\gamma}\}$ and $\{m_{k,\gamma}'\}$ on $\Omega^*(L,\Lambda)$.
 Then there exists a $1$-parameter family of compatible almost complex structures $\{J_t\}_{t\in [0,1]}$ connecting $J,J'$. Moreover, Fukaya proved 
 \begin{thm} \cite{F1} \label{2899}
    There exists an pseudo-isotopy $(\{m^t_{k,\gamma}\},\{c^t_{k,\gamma}\})$ between the two $A_{\infty}$ structures induced by the almost complex structures $J$ and $J'$. 
 \end{thm}
 Thus the $A_{\infty}$ structure on $H^*(L,\Lambda)$ can be viewed as an symplectic invariant and does not depend on the choices of almost complex structures up to pseudo-isotopies. By constructing compatible forgetful maps, we have the analogue of divisor axiom in Gromov-Witten theory:
   \begin{thm} \cite{F2} \label{4053} Let $k\geq 0$ and $b, x_0,\cdots, x_k\in H^1(L,\Lambda_+ )$, then 
   \begin{align*}
      \sum_{m_0+\cdots+m_k=m}& m_{k+m,\gamma}(b^{\otimes m_0},x_1,b^{\otimes m_1},\cdots, b^{\otimes m_{k-1}},x_k,b^{\otimes m_k}) \\
       =& \frac{1}{m!}\langle \partial \gamma,b\rangle^m m_k(x_1,\cdots,x_k).
    \end{align*}
   \end{thm}
 Similar argument to Theorem \ref{4053}, we have the divisor axiom for both $m^t_{k,\gamma}$ and $c^t_{k,\gamma}$. 
     \begin{thm} \label{4054} Given a pseudo-isotopy $(m^t_{k,\gamma},c^t_{k,\gamma})$ between $A_{\infty}$-structures, then similar statement of Theorem \ref{4053} holds for $m^t_{k,\gamma}$ and $c^{t}_{k,\gamma}$.  
     \end{thm}

      Assume that $\mathcal{M}_{\gamma}(X_{\vartheta},L_{u_0})$ is non-empty and $L_u$ is a nearby smooth special Lagrangian with the same property. Write $\partial \gamma= \sum_i a_i e_i$, for some $a_i\in \mathbb{Z}$, then 
     \begin{align*}
       \sum_i a_i (f_i(u)-f_i(u_0))= a_i(\int_{\gamma_u}\mbox{Im}\Omega-\int_{\gamma_{u_0}}\mbox{Im}\Omega) =0-0=0. 
     \end{align*} Thus, we reach the following proposition:
     \begin{prop} \label{21}
        Assume that $L_u$ is a nearby smooth special Lagrangian such that $\mathcal{M}_{\gamma}(X,L_{u_0})$ and $\mathcal{M}_{\gamma}(X,L_u)$ are non-empty for some relative class $\gamma$, then $u$ falls on an affine hyperplane passing through $u_0$.
     \end{prop} 
       
       We will say an affine line $l$ in Proposition \ref{21} is labeled by $\gamma$ (locally) and denoted it by $l_{\gamma}$. Notice that $l_{\gamma}$ has a natural orientation due to the fact that $|Z_{\gamma}|$ is monotone on $l_{\gamma}$. Since $H_2(X,L_u)$ is countable, we have the following corollary:
     \begin{cor} \label{49}
        For any generic point $u\in B_0$, the special Lagrangian submanifold $L_u$ bounds no holomorphic discs.
     \end{cor}
      \begin{lem}
              For any $b_1,b_2\in H^1(L_u)$ and $b_1\neq b_2$, then $b_1$ is not gauge equivalent to $b_2$. 
            \end{lem}
            \begin{proof}
                The lemma follows from Corollary \ref{49} when $L_u$ is a generic torus fibre.  Together with Proposition \ref{35} and Fukaya's trick prove the lemma for general case.
            \end{proof}
     
           The Proposition \ref{21} together with Gromov compactness theorem gives the follow corollary:
     
     \begin{cor} \label{29}
       Fix $\lambda>0$, then the set
        \begin{align*}
                    L_{\lambda}:= \{u\in B_0| L_u \mbox{ bounds a holomorphic discs with area less than $\lambda$. }\}
                  \end{align*}
        is locally a closed subset of finite union of
       walls of first kind with energy less than $\lambda$.
     \end{cor}
     The central charge function (\ref{4010}) plays an important roles for holomorphic discs with special Lagrangian boundary condition in hyperK\"ahler manifolds. 
     
       \begin{lem}\cite{L4}
                If $\mathcal{M}_{\gamma}(\mathfrak{X}^{[\omega]},L)\neq \emptyset$, then $\mathcal{M}_{\gamma}(\mathfrak{X}^{[\omega]},L)=\mathcal{M}_{\gamma}(X_{\vartheta},L)$, $\vartheta=\mbox{Arg}Z_{\gamma}+\frac{\pi}{2}$ as topological spaces. 
             \end{lem}

    Assume $\gamma\in H_2(X,L_u)$ and $\mathcal{M}_{\gamma}(X,L_u)$ has non-trivial real codimension one boundary. Namely, there exists $\gamma_1,\gamma_2\in H_2(X,L_u)$ such that $\gamma=\gamma_1+\gamma_2$ and $u$ satisfies the equation
        \begin{align*}
           \mbox{Arg}\big(\int_{\gamma_1}\mbox{Im}\Omega+i\omega\big)= \mbox{Arg}\big(\int_{\gamma_2}\mbox{Im}\Omega+i\omega\big)=0,
        \end{align*} which cuts out finitely many points on $B_0$ \cite{L4}.

   For our purpose, we need the following Fukaya's trick: 
   Given a reference point $p\in B_0$ and a path $\phi$ in a neighborhood $U$ of $p$ such that $\phi(0)=u_0, \phi(1)=u_1$. Assume that $U$ is small enough so that there exists a $2$-parameter family of fibrewise preserving diffeomorphisms $\psi_{s,t}$ of the K3 surface satisfying 
      \begin{enumerate}
         \item $\psi_{0,t}=\mbox{Id}$ and $\psi_{1,t}(L_{\phi(t)})=L_p$. Thus, $\psi_{s,t}(L_{\phi(t)})$ induces a path from $\phi(t)$ to $p$ for fixed $t$. 
         \item Let $J$ be the almost complex structure of the K3 surface $X$. The complex structure $(\psi_{s,t})_*J$ is tame with respect to the symplectic form $\omega$. 
         \item The diffeomorphism $\psi_{s,t}$ is an identity away from $L_u$, where $u$ is outside of the neighborhood $U$. 
      \end{enumerate} Then for each relative class $\gamma\in H_2(X,L_{p})$, the two moduli spaces of holomorphic discs \footnote{Here there is a natural identification $H_2(X,L_{u_0})$ between $H_2(X,L_{\phi(t)})$ via ${\psi_{1,t}}_*$, which is the parallel transport of the Jacobian fibration.}
      \begin{align}\label{3049}
       \mathcal{M}_{\gamma}(X,& L_{\phi(t)})  \rightarrow \mathcal{M}_{\gamma}(X, L_{p}) \notag  \\
           &\alpha \longmapsto  (\psi_{1,t})_* \circ \alpha
      \end{align} are naturally identified. Here the later moduli space is with respect to the complex structure $J$, while the former one is respect to the complex structure $(\psi_{1,t})_*J$.  Thus via the choice of $2$-parameter family of diffeomorphisms $\phi_{s,t}$, we can have a $1$-parameter family of (almost) complex structures $\{J_t=(\psi_{1,t})_*J\}_{t\in [0,1]}$. Write the symplectic affine coordinates of $u$ associate to the basis $u_1,\cdots,u_2$, where $u_i=\int_{\bar{e_i}_u}\omega$. Since $J_t$ is a small perturbation of complex structure of $J$, Corollary \ref{22} gives a pseudo-isotopy between 
              \begin{align} \label{7878}
                 \big(\Omega^*(L_{p},\Lambda),\sum_{\gamma}m_{k,\gamma}T^{\int_{\gamma}\omega}\big) \mbox{ and } \big(\Omega^*(L_{p},\Lambda), m_{k,\gamma}T^{\int_{\gamma}\omega-\sum u_i ( \partial\gamma,e_i)}\big), 
              \end{align} where the former is computed using the almost complex structure $(\psi_{1,0})_*J$ while the later is computed using $(\psi_{1,1})_*J$. Here we use the notation $\partial\gamma=\sum_k (\partial\gamma,e_i) e_i$ and $(,)$ denotes the natural pairing between $H_1(L)$ and $H^1(L)$.
   In particular, it induces an isomorphism on the Maurer-Cartan spaces of the canonical models \footnote{Notice that the $A_{\infty}$ structures are slightly different from the one on $H^*(L_{u_0},\Lambda),H^*(L_{u_1},\Lambda)$ due to the flux.},
     \begin{align*}
        F^{can}_{(\phi,p)}: H^1(L_{p},\Lambda_+)=\mathcal{MC}(H^*(L_{p},\Lambda_+)) \cong \mathcal{MC}(H^*(L_{p},\Lambda_+))= H^1(L_{p},\Lambda_+) .
     \end{align*} Here the first $\mathcal{MC}(H^*(L_p,\Lambda_+))$ is the Maurer-Cartan space for the $A_{\infty}$ algebra of the former one in (\ref{7878}) and the second one is the Maurer-Cartan space for the later one in (\ref{7878}).  
First of all, the isomorphism $F^{can}_{(\phi,p)}$ is independent of the choices made to construct the Kuranishi structures and perturbed multi-sections. Indeed, the different choices will induce pseudo-isotopies of the pseudo-isotopies and $F^{can}_{(\phi,p)}$ is well-defined by Theorem \ref{35}.  
     
     For a different choice of the $2$-parameter family of the diffeomorphisms $\psi_{s,t}$ and $\psi_{s,t}'$, the resulting isomorphism $F^{can}_{(\phi,p)}$ and $(F^{can}_{(\phi,p)})'$ on the Maurer-Cartan spaces are the same due to the natural isomorphism (\ref{3049}). Moreover, for a different choice of reference point $p'$, the resulting isomorphisms on the Maurer-Cartan spaces are related by 
       \begin{align*}
           F^{can}_{(\phi,p)}=T_{p',p}F^{can}_{(\phi,p')}.
       \end{align*} 
     
      The following proposition follows directly from the construction of $F^{can}_{\phi}$. 
   \begin{prop} \label{32}
       Fix $\lambda>0$. Assume that $L_{\phi(t)}, t\in [0,1]$ does not bound any holomorphic discs of relative $\gamma$ such that  $|Z_{\gamma}(p)|<\lambda$, then the isomorphism 
         \begin{align*}
            F^{can}_{(\phi,p)}=\mbox{Id}_{H^1(L,\Lambda_+)} \hspace{4mm} (\mbox{mod }T^{\lambda}).
         \end{align*}

     \end{prop}

    The following is the key theorem to prove the wall-crossing formula of open Gromov-Witten invariants in the later section. 
    
    \begin{thm} \cite{F1} \cite{T4} \label{102}
      Let $\phi_0,\phi_1$ be two paths with the same end points and homotopic to each other, within a small enough open neighborhood $U\ni p$ in $B_0$. Then $F^{can}_{(\phi_0,p)}=F^{can}_{(\phi_1,p)}$. In particular, if $\phi$ is a loop contractible in an small enough open set $U\subseteq B_0$. Then $F^{can}_{(\phi,p)}=Id$. 
    \end{thm}
     \begin{proof}
       Assume that  $\Phi=\Phi(s,t)$ is a homotopy between $\phi_0$ and $\phi_1$, i.e.,
         \begin{align*}
             &\Phi(0,t)=\phi_0(t)   \\
             &\Phi(1,t)=\phi_1(t).
         \end{align*} 
       Then the homotopy $\Phi$ induces a pseudo-isotopy of pseudo-isotopy. Then the theorem follows from Theorem \ref{47}. 
   \end{proof}

  \begin{rmk}
  From symplectic geometric point of view, Theorem \ref{102} is more natural than the Kontsevich-Soibelman wall-crossing formula. This can be view as the mathematical replacement of the argument in Cecotti-Vafa \cite{CV}.
  \end{rmk}

\section{Open Gromov-Witten Invariants on K3 Surfaces}  
  \subsection{Open Gromov-Witten Invariants} \label{4003}
      In this section, we will restrict ourselves to elliptic K3 surfaces\footnote{See \cite{L7} for the open Gromov-Witten invariants with rigid special Lagrangian boundary conditions.}. 
       Given an elliptic K3 surface $X\rightarrow B$ and a K\"ahler class $[\omega]$, there exists an $S^1$-family of hyperK\"ahler structures such that the fibration $X\rightarrow B$ become special Lagrangian fibration. We will denote the corresponding hyperK\"ahler manifold by $X_{\vartheta}, \vartheta\in S^1$.

      Similar to the story of tropical discs, we have an analogue definition for holomorphic discs:
    \begin{definition}
       Let $\gamma$ be a relative class with $\partial \gamma\neq 0$, locally we define the locus $W'_{\gamma}$ to be 
          \begin{align*}
             W'_{\gamma}=\bigcup_{\substack{\gamma_1+\gamma_2=q\gamma\\ q\in \mathbb{Q}, 0\leq q \leq 1}} W'_{\gamma_1,\gamma_2},
          \end{align*} where
          \begin{align*}
             W'_{\gamma_1,\gamma_2}=\{u\in B| \mbox{Arg}Z_{\gamma_1}=\mbox{Arg}Z_{\gamma_2} \mbox{ and there exist holomorphic discs}\\ \mbox{ of relative class $\gamma_1,\gamma_2$ ends on $L_u$ with $\langle \gamma_1,\gamma_2\rangle\neq 0$.}\}.
          \end{align*}
    \end{definition}
    Then locally $W'_{\gamma}$ is a closed subset of the real codimension one locus cut out by 
       \begin{align*}
          \mbox{Arg}Z_{\gamma_1}=\mbox{Arg}Z_{\gamma_2}.
       \end{align*}
    Now assume $\gamma\in H_2(X,L_u)$ is primitive such that $u\neq W'_{\gamma}$, then the moduli space
    \begin{align*}
     \mathcal{M}_{\gamma}(\mathfrak{X},L_u)=\cup_{\vartheta\in S^1}\mathcal{M}_{\gamma}(X_{\vartheta},L_u)
    \end{align*} is compact without boundary (thus no real codimension one boundary)\cite{L4}. In particular, the moduli space admits a virtual fundamental class $[\mathcal{M}_{\gamma}(\mathfrak{X},L_u)]^{vir}$\cite{FOOO}\cite{L4}. We may define the open Gromov-Witten invariants as follows:
       \begin{align*}
          \tilde{\Omega}(\gamma;u):=\int_{[\mathcal{M}_{\gamma}(\mathfrak{X},L_u)]^{vir}}1.
       \end{align*} In general, we define the open Gromov-Witten invariants via the smooth correspondence in \cite{L4},
       \begin{align*} 
          \tilde{\Omega}(\gamma;u):=Corr_*(\mathcal{M}_{\gamma}(\mathfrak{X},L_u);tri, tri)(1).
       \end{align*} We will refer the readers to \cite{F1} for the definition and details about the smooth correspondences. The following are some properties of the open Gromov-Witten invariants $\tilde{\Omega}(\gamma;u)$.

    \begin{prop} \label{3012} \cite{L4}
          Assume that $u\notin W'_{\gamma}$ then $\tilde{\Omega}(\gamma;u)$ is well-defined. Moreover, we have the following properties:
          \begin{enumerate}
             \item The invariant $\tilde{\Omega}(\gamma;u)$ is independent of choice of the K\"ahler class $[\omega]$. 
             \item If there exists a path connecting $u,u'\in B_0$ which does not pass through $W'_{\gamma}$, then 
                       \begin{align*}
                           \tilde{\Omega}(\gamma;u)=\tilde{\Omega}(\gamma;u').
                       \end{align*} In particular, $W'_{\gamma}$ locally divides $B_0$ into chambers and  $\tilde{\Omega}(\gamma;u)$ is locally constant in each chamber. 
             \item Let $u\in W'_{\gamma}$ generic and let $u_+,u_-$ be on the different sides of $W'_{\gamma}$ near $u$. Assume that there exists no $\gamma_i\in H_2(X,L_u)$ such that $\sum_{i}\gamma_i=\gamma$, $\mbox{Arg}Z_{\gamma_i}(u)=\mbox{Arg}Z_{\gamma}(u)$ and $\tilde{\Omega}(\gamma_i;u)\neq 0$, $\langle \gamma_i,\gamma_{i'}\rangle \neq 0$ for some $i,i'$, then 
             \begin{align*}
                 \tilde{\Omega}(\gamma;u_+)=\tilde{\Omega}(\gamma;u_-).
             \end{align*}        
             \item (Reality condition) 
             $\tilde{\Omega}(\gamma;u)=\tilde{\Omega}(-\gamma;u)$.
          \end{enumerate}
           \end{prop}
       The open Gromov-Witten invariants near an $I_1$-type singular fibre is calculated via Theorem \ref{4020} and a cobordism argument.      
      \begin{thm}\label{3020} \cite{L4}
         Let $L_{u_0}$ be an type $I_1$-singular fibre. Then there exists a sequence neighborhood $U_d$ of $u_0$, $U_{d+1}\subseteq U_{d}$ such that for each $u\in U_d$,
            \begin{align*}
               \tilde{\Omega}(\gamma;u)=\begin{cases}
                  \frac{(-1)^{d-1}}{d^2}, & \gamma=d\gamma_e \\
                  0, & \mbox{otherwise},
               \end{cases}
            \end{align*} where $\gamma_e$ is the relative class of the Lefschetz thimble associate to $L_{u_0}$. 
      \end{thm}     
           
    One of the application of the open Gromov-Witten invariants is the weak correspondence theorem, which gives a sufficient condition for the existence of tropical discs.
    \begin{thm}\cite{L4}
       Let $u\in B_0$ and $\gamma\in H_2(X,L_u)$ such that $u\notin W'_{\gamma}$ and $\tilde{\Omega}(\gamma;u)\neq 0$. Then there exists a tropical disc $(\phi,T,w)$ such that $[\phi]=\gamma$. 
    \end{thm}

         \subsection{Local Model: Focus-Focus Singularity}
           \begin{definition}
            The Ooguri-Vafa space $X_{OV}$ is an elliptic fibration over a unit disc with a unique singular fibre, a single node rational curve, over the origin. 
           \end{definition}      
            From the explicit coordinate description, there exists a natural $S^1$-action which preserves the complex structure. There exists a hyperK\"ahler metric on $X_{OV}$, realized as a periodic Gibbons-Hawking ansatz, thus is $S^1$-invariant. Conversely, any $S^1$-invariant hyperK\"ahler metric on the same underlying space is in this form \cite{GW}.

            For a fixed $\vartheta\in S^1$, the complex affine coordinate associated to the special Lagrangian fibration in $(X_{OV})_{\vartheta}$ gives rise to an affine structure on $B$ with a singularity at the origin. The monodromy of the affine structure around the origin is conjugate to $\bigl(
            \begin{smallmatrix}
              1 & 1\\
              0 & 1
            \end{smallmatrix} \bigr)$. There are two affine rays $l_{\pm}$ emanating from the origin in the monodromy invariant direction. Chan studied the holomorphic discs in the Ooguri-Vafa space using the maximal principle and the $S^1$-invariance of the hyperK\"ahler structure. 
            \begin{prop} \cite{C} \label{33}
              \begin{enumerate}
                 \item There exist a simple holomorphic disc with boundary in each special Lagrangian torus fibre over $l_{\pm}$. Moreover, the image of the holomorphic disc is smooth.
                 \item The above holomorphic discs together with their multiple covers are the only holomorphic discs in $(X_{OV})_{\vartheta}$.
               \end{enumerate}
            \end{prop}
           \begin{rmk}
              For each $u\in l_{\pm}$, the affine segment (with multiplicity) from the origin to $u$ gives a tropical disc with stop at $u$ which corresponds to the simple holomorphic disc (or its multiple cover) in Proposition \ref{33}.
           \end{rmk} 
          As the phase $\vartheta$ goes around $S^1$, the two affine rays also rotate around the origin counterclockwisely and every point will be exactly swept once. In other words, every torus fibre bounds exactly a simple disc (up to orientation) which is holomorphic to some complex structure $(X_{OV})_{\pm\vartheta}$. 
            
         The holomorphic discs in Proposition \ref{33} falls in a small neighborhood $U_{OV}$ of the singular point of the singular fibre. One can find a $1$-parameter family of $S^1$invariant hyperK\"ahler structures connecting the one on $U_{OV}$ and an	 open set of $T\mathbb{P}^1$ with Eguchi-Hanson metric. The later admits an anti-symplectic involution and every holomorphic disc can be doubled to a real rational curve. One than can use the localization to compute the open Gromov-Witten invariants.
       \begin{thm} \cite{L4} \label{4020}
            Let $X=X_{OV}$ be the Ooguri-Vafa space and $\gamma_e$ to be the relative class of the Lefschetz thimble.
            Then the open Gromov-Witten invariants on $X$ is calculated:
              \begin{align*}
                 \tilde{\Omega}(\gamma;u)=\begin{cases}
                 \frac{(-1)^{d-1}}{d^2}, \mbox{ if }\gamma=d\gamma_e, d\in \mathbb{Z}. \\
                 0, \mbox{ otherwise}.
                 \end{cases}
              \end{align*} 
       \end{thm}

           Let $p_{\pm}\in l_{\pm}$ and $U_{\pm}$ is a small neighborhood of $p_{\pm}$. 
           Let $u^{\pm}_1,u^{\pm}_2\in U_{\pm}$ be on the different side of $l_{\pm}$ as in Figure \ref{fig:95}. Choose paths $\phi_{\pm}:[0,1]\rightarrow B_0$ such that $\phi_{\pm}(0)=u^{\pm}_1$, $\phi_{\pm}(1)=u^{\pm}_2$ and intersect $l_{\pm}$ exactly once and transversally. The family of Lagrangians induces a pseudo-isotopy between the $A_{\infty}$ structures on $H^*(L_{u_1})$ and $H^*(L_{u_2})$. In particular, the pseudo-isotopy induces an isomorphism $F^{can}_{(\phi_{\pm},p)}$ between $\mathcal{MC}(L_{u_1})$ and $\mathcal{MC}(L_{u_2})$. The explicit expression of the isomorphism follows directly from Theorem \ref{4020}, Theorem \ref{3005} and equation (\ref{4030}).
          \begin{thm} \label{42}
          Let $e_1,e_2$ be an symplectic integral basis of $H^1(L_{p_{\pm}},\mathbb{Z})$ with $\langle e_1,e_2\rangle=1$ and $-e_2$ is the Poincaré\'e dual of the vanishing cycle $\gamma_e$. 
            \begin{align}
               &F^{can}_{(\phi_{\pm},p_{\pm})}(z_1)=z_1  \notag  	\\
               &F^{can}_{(\phi_{\pm},p_{\pm})}(z_2)=z_2(1+T^{\pm Z_{\gamma_e}(p_{\pm})}z_1^{\pm 1})^{\mp 1},
            \end{align}  where $\gamma_e$ is the relative class of Lefschetz thimble.
         \end{thm}

       \begin{figure}
                      \begin{center}
                      \includegraphics[height=3in,width=6in]{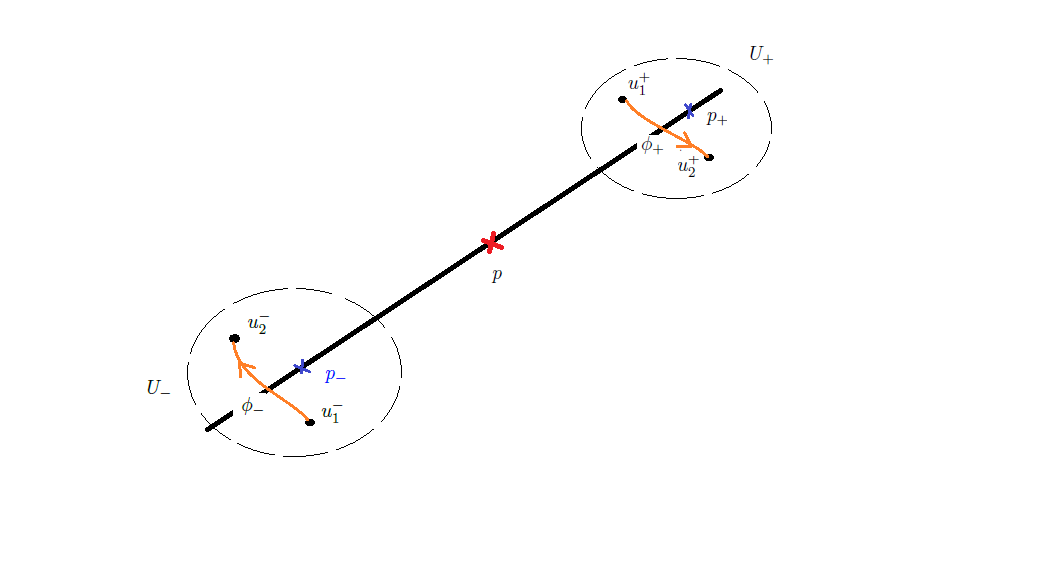}
                      \caption{On the base of Ooguri-Vafa space}
                       \label{fig:95}
                      \end{center}
                      \end{figure}
           
  
  \subsection{Wall-Crossing of Maurer-Cartan Elements}
    Now we want to understand the wall-crossing phenomenon of the open Gromov-Witten invariants. Given a basis $e_1, e_2\in H^1(L_p,\mathbb{Z})$ such that $\langle e_1,e_2\rangle=1$,
    there exists a natural symplectic $2$-form on $H^1(L_p;\Lambda_+)$ given by
            \begin{align*}
             \varpi=  \frac{dz_1}{z_1}\wedge \frac{dz_2}{z_2}.
            \end{align*}
  \begin{lem}
     The symplectic $2$-form $\varpi$ is independent of choice of the basis $e_i\in H^1(L_p,\mathbb{Z})$. 
  \end{lem}
   \begin{proof}
      Since any two basis are related by a transform $A=GL(2,\mathbb{Z})$, which can be decomposed into product of elementary transformation. It is straightforward to check that $\varpi$ is the same if the two basis is differed by an elementary transformation. 
   \end{proof}
  
     Now assume that $X\rightarrow B$ is a K3 surface with special Lagrangian torus fibration. Let $p\in B_0$ and $U\subseteq B_0$ is a small enough neighborhood of $p$. Let $u_0,u_1\in U$. Let $\phi$ be a path in $U$ connecting $u_0,u_1$ such that "Fukaya's trick" applies (see Section \ref{4006}). 
    We will use the same notation $e_1,e_2$ for their parallel transport along $\phi$. Write $b=x_1e_1+x_2e_2\in H^1(L_{u},\Lambda_+)$. Using equation (140) of \cite{F1}, we can compute the form of $F^{can}_{(\phi,p)}$ with respect the above coordinates:
              \begin{align*}
                     F^{can}_{(\phi,p)}:  &H^1(L_{u_0},\Lambda_+) \longrightarrow H^1(L_{u_1},\Lambda_+)\\
                      & b=x_1e_1+x_2e_2 \mapsto (F^{can}_{(\phi,p)}(b))_1e_1+(F^{can}_{(\phi,p)}(b))_2e_2
              \end{align*} and gives a transformation $x_k \mapsto (F^{can}_{(\phi,p)}(b))_k$, $k=1,2$.    
     %
     %
        Motivated by mirror symmetry, it is more natural to consider $z_k=\exp{(x_k)}$ and transformation $F^{can}_{(\phi,p)}$ acts on $z_k$ by $z_k\mapsto \exp (F^{can}_{(\phi,p)}(b))_k$. This defines an automorphism of the Tate algebra \cite{T4}. It is straight-forward to check that the transformation induced by $F^{can}_{(\phi,p)}$ is independent of the choice of the integral basis $e_1,e_2\in H^1(L_p,\mathbb{Z})$. 
   
    Let $\mathcal{M}^t_{k,\gamma}$ be the moduli spaces of stable discs (holomorphic with respect to $J_t$) with boundary on $L_p$ and $k$ boundary marked points. We will denote the family moduli space by $\tilde{\mathcal{M}}_{k,\gamma}=\cup_{t\in [0,1]}\mathcal{M}^t_{k,\gamma}$. Let $tri$ denotes the trivial map $\tilde{\mathcal{M}}_{k,\gamma}\rightarrow pt$ to a point. Let $ev_k$ be the evaluation map at the $k$-th boundary marked point and $ev_t$ be the evaluation map to $[0,1]$. Recall that the $c^t_{0,\gamma}$ of the pseudo-isotopy in Theorem \ref{2899} is constructed via smooth correspondences
     \begin{align*}
        Corr_*(\tilde{\mathcal{M}}_{1,\gamma};tri,(ev_0,ev_t))(1)=m^t_{0,\gamma}+dt\wedge c^t_{0,\gamma}.
     \end{align*} Fukaya further considered
     \begin{align*}
     Corr_*(\tilde{\mathcal{M}}_{0,\gamma};tri,ev_t)(1) =m^t_{-1,\gamma}+   c^t_{-1,\gamma}dt
     \end{align*} for understanding the wall-crossing phenomenon of the holomorphic discs \cite{F1}. With the above notation, we have the following lemma to help understand further the form of $F^{can}_{(\phi,p)}$:
          \begin{lem}\label{27} Let $b\in \Omega^1(L_p)$ be a closed $1$-form and $\gamma$ be a relative class, then 
            \begin{align} 
               \int_{L\times [t_0,t_1]} c^t_{0,\gamma}\wedge b dt= \int_{[t_0,t_1]}c^t_{-1,\gamma}dt\int_{\partial \gamma}b,
            \end{align} for any $[t_0,t_1]\subseteq [0,1]$. 
          \end{lem}
          \begin{proof}
          Let  $\mathfrak{forget}$ denotes the forgetful map $\tilde{\mathcal{M}}_{1,\gamma}\rightarrow \tilde{\mathcal{M}}_{0,\gamma}$. In terms of smooth correspondence, the left hand side of (\ref{27}) is 
            \begin{align*}
               & \int_{L\times [t_0,t_1]} Corr_*(\tilde{\mathcal{M}}_{1,\gamma};tri, (ev_0,ev_t))(1)\wedge b \\
             =& Corr_*(\tilde{\mathcal{M}}_{1,\gamma};(ev_0,ev_t),tri)(b) \\
             =& Corr_*(\tilde{\mathcal{M}}_{1,\gamma};(ev_0,ev_t),tri\circ \mathfrak{forget})(b)\\
            =& Corr_*(\tilde{\mathcal{M}}_{0,\gamma};tri,tri)(1) \cdot \int_{\gamma}b\\
            =&\int_{[t_0,t_1]}c^t_{-1,\gamma}dt\int_{\partial \gamma}b.
            \end{align*} The first equality is the composition formula of smooth correspondences (Lemma 4.3\cite{F1}). The third equality comes from integration along the fibre of the forgetful map and the compatibility of the forgetful map \cite{F1}. 
            
          \end{proof}
           Lemma \ref{27} implies that for any smooth function $b'$on $L_p$, we have 
             \begin{align*}
               \int_{L_p} b' (dc^t_{0,\gamma}) =-\int_{L_p} c^t_{0,\gamma}\wedge db'= 0.
             \end{align*}
        In particular, we have                
 \begin{cor} \label{4056} $c^t_{0,\gamma}$ is $d$-closed $1$-form and
            $ \langle \partial \gamma, c^t_{0,\gamma}\rangle=0$
        \end{cor}  
        
    Now we can state the key theorem to compute the open Gromov-Witten invariants $\tilde{\Omega}'(\gamma;u)$ defined in Section \ref{4055}.    
      
    The following definition is natural from the Torelli theorem of K3 surfaces.  
  \begin{definition} \label{9002}
     Let $X$ be a hyperK\"ahler surface with elliptic fibration and admits only type $I_1$ singular fibres. Then 
       \begin{enumerate}
           \item we say that $X$ satisfies the generic condition (*) if $\Omega:H_2(X,\mathbb{Z})\rightarrow \mathbb{C}$ is injective.
           \item we say that $X$ satisfies the generic condition $(*)_{u,\gamma}$ if $\gamma'\in H_2(X,L_u)$ represented by a holomorphic disc in $X_{\vartheta}$ for some $\vartheta$ and $Z_{\gamma'}=tZ_{\gamma}$ in a neighborhood of $u$ for some $t\in \mathbb{Q}$, then $t>1$. 
           \item we say that $X$ satisfies the generic condition $(*)_{\lambda}$ for some $\lambda>0$ if $X$ satisfies the generic condition $(*)_{u,\gamma}$ for each pair $(u,\gamma)$ such that $|Z_{\gamma}(u)|<\lambda$. 
       \end{enumerate}
  \end{definition} 
  \begin{rmk} \label{445}
  \begin{enumerate}
     \item  By the Torelli theorem of K3 surfaces, the K3 surface satisfies the generic condition (*) is indeed generic in the period domain. By the Torelli theorem of K3 surfaces and Gromov compactness theorem, the generic condition $(*)_{\lambda}$ is an open condition in the period domain. 
     \item 
          It is easy to see that generic condition (*) implies generic condition $(*)_{u,\gamma}$ for all pairs $(u,\gamma)$.
      \item  If $X$ satisfies the generic condition (*), then $l_{\gamma}\neq l_{\gamma'}$ for a generic $\vartheta$ unless $\gamma=k\gamma'$ for some $k\in \mathbb{Q}$.    
           Similarly, if $X$ satisfies generic condition $(*)_{u,\gamma}$ and $l_{\gamma'}=l_{\gamma}$ for a generic $\vartheta$, then either $\gamma'=k\gamma$ for some $k\in \mathbb{Q}$ for $|Z_{\gamma'}(u)|>|Z_{\gamma}(u)|$. 
   \end{enumerate}
  \end{rmk}

      Let $u_{\pm}$ be two points on the different sides of $l_{\gamma}$ and $\phi$ is a path connecting $u_{\pm}$. We may take $u_{\pm}$ be close enough to $l_{\gamma}$ such that there are only relative classes $d\gamma$, where $d\in\mathbb{N}$ and $\gamma$ primitive, can be realized as holomorphic discs with boundary on $L_{\phi(t)}$ for some $t$ and with symplectic area less than $\lambda$ for a given $\lambda>0$.
   \begin{thm} \label{101}
    The transformation $F^{can}_{(\phi,p)}$ is of the form 
           \begin{align*}
               z^{\partial\gamma'}\mapsto z^{\partial \gamma'}f_{\gamma}^{\langle \gamma',\gamma\rangle} \hspace{3mm} (\mbox{ mod $T^{\lambda}$}),
           \end{align*} for some $f_{\gamma}\in 1+\Lambda[[z^{\partial\gamma}]]$. Here we use the notation that
                      \begin{align*}
                        z^{\partial \gamma}=z_1^{( \partial\gamma, e_1)}z_{2}^{(\partial\gamma,e_{2})}.
                      \end{align*}
   \end{thm}
  \begin{proof}
     By Proposition \ref{35}, we have 
          \begin{align}\label{4998}
             F^{can}_{(\phi,p)}(b)=\sum_{k\geq 1,d\geq 1}\mathfrak{c}(k,d\gamma)(b,\cdots,b)T^{d\omega(\gamma)},
          \end{align} where $\mathfrak{c}(k,\gamma)$ is defined in Theorem \ref{4050}. Let $\mathfrak{c}'(k,d\gamma)(b)$ denote the sum of the terms in (\ref{4998}) which are corresponding to the trees with $k$ internal vertices.
          First we will show that 
          \begin{align} \label{4052}
          \mathfrak{c}'(k,d\gamma)(b)=0
          \end{align}
           (before integrating over the time allocation) by induction on $k$ when $k\geq 2$. For $k=2$, we have 
            \begin{align*}
                \mathfrak{c}'(2,d\gamma)(b)&= \sum_{l_1,l_2\geq 0} \sum_{i=1}^{l_1-1} \int_{\tau_1>\tau_2} c^{\tau_1}_{l_1,\gamma}(\underbrace{b,\cdots, b}_\text{$i$ copies}, c^{\tau_2}_{l_2,\gamma}(b,\cdots,b),\cdots,b).
            \end{align*}
          By Theorem \ref{4053}, we have 
           \begin{align*}
              &\sum_{i=1}^{l_1-1}c^{\tau_1}_{l_1,\gamma}(\underbrace{b,\cdots, b}_\text{$i$ copies}, c^{\tau_2}_{l_2,\gamma}(b,\cdots,b),\cdots,b) \\
             =& \frac{\langle \partial \gamma,b\rangle^{l_1-1}}{(l_1-1)!}c^{\tau_1}_{1,\gamma}(c^{\tau_2}_{l_2,\gamma}(b,\cdots,b)) \\  
             =& \frac{\langle \partial \gamma,b\rangle^{l_1+l_2-1}}{(l_1+l_2-1)!} c^{\tau_1}_{1,\gamma}(c^{\tau_2}_{0,\gamma}) \\
             =& \frac{\langle \partial \gamma,b\rangle^{l_1+l_2-1}}{(l_1+l_2-1)!} \langle \partial \gamma,c^{\tau_2}_{0,\gamma}\rangle c^{\tau_1}_{0,\gamma}=0.
           \end{align*} Let last equality follows from Corollary \ref{4056}. Assume that $e_1^*=c^{-1}\partial \gamma$ is primitive in $H_1(L_p,\mathbb{Z})$ for some constant $c$. Extend it to a symplectic integral basis $e_1^*,e_2^*$ and let $e_1,e_{2}$ be the corresponding dual basis in $H^1(L_p,\mathbb{Z})$.  For $k>2$, the equation (\ref{4052}) follows from the induction hypothesis and the divisor axiom. 
          For $k=1$, the only tree which contributes to $\mathfrak{c}'(1,d\gamma)(b,\cdots,b)$ has one interior vertex $v$ and $l+1$ exterior vertices (including the root), $l\geq 0$ or only one vertex.
           Thus 
                 \begin{align*}
                    \mathfrak{c}'(1,d\gamma)(b)&=\sum_{l\geq 0}\int_0^1 c^{\tau}_{l,d\gamma}(b,\cdots,b) d\tau\\
                    &= \sum_{l\geq 0}\frac{\langle d\partial \gamma,b\rangle^l}{l!} \int_0^1 c^{\tau}_{0,d\gamma}d\tau \\
                    &= (z^{\partial \gamma})^d \int_0^1 c^{\tau}_{0,d\gamma}d\tau
                 \end{align*}
      Together with Corollary \ref{4056}, the transformation given by $F^{can}_{(\phi,p)}$ becomes
            \begin{align*}
              z_i \rightarrow z_i f, \mbox{ where } f= 1+\begin{cases} 0& \text{if $i=1$}\\
                                  \mbox{ power series in $z_1^c$}& \text{if $i=2$,}
                                                                            \end{cases}
            \end{align*} 
     
  \end{proof}
      Fix $\lambda>0$. There are finitely many of affine hyperplanes $l_{\gamma}$ passing through $U$ such that there exists $\gamma\in H_2(X,L_u)$ with $\mathcal{M}_{\gamma}(X,L_u)\neq \emptyset$ and $|Z_{\gamma}(p)|<\lambda$ for some $u\in U$. A direct consequence of  Theorem \ref{101} is the following theorem. 
  \begin{thm}
     Let $u_+,u_-$ be in a small enough neighborhood $U\ni p$ and a path $\phi$ connecting $u_+,u_-$ falls in $U$.
     Then we have $(F^{can}_{(\phi,p)})^*\varpi=\varpi$.
  \end{thm}

\begin{rmk} \begin{enumerate}
     \item  We expect that the Theorem \ref{101} is true for special Lagrangian surface of higher genus as well. However, our argument does not apply directly. 
     \item Similar result in a different setting and assumption also appears in the work of Seidel (Section 11 \cite{S4}).
\end{enumerate}
\end{rmk}
 
   \subsection{Locally Around Type $I_1$ Singular Fibre}
From this section we will focus on the case when $X$ is a K3 surface with special Lagrangian fibration and $L$ is a special Lagrangian torus fibre. Assume moreover that each of the singular fibre has a single simple node and we want to compute $F^{can}_{\phi}$ for a path $\phi$ connecting $u_1,u_2\in B_0$ near a singularity $0\in B$ of the affine structure.

\begin{thm} \label{45}
  Given $\lambda>0$ and a generic $\vartheta\in S^1$, there exists a neighborhood $\mathcal{U}_{\lambda}\ni 0$ such that for any path $\phi$ in $\mathcal{U}_{\lambda}$ connecting $u_1,u_2$, then
   \begin{enumerate}
      \item if $\phi$ does not intersect $l_{\pm}$, then
              $F^{can}_{\phi}\cong \mbox{Id} \mbox{ ( mod $T^{\lambda}$)}$.
      \item if $\phi$ is homotopy to $\phi_{\pm}$ within $\mathcal{U}_{\lambda}$, then
      \begin{align*}
              \begin{cases}
               F^{can}_{(\phi,p_{\pm})}(z_1)=z_1  \\
               F^{can}_{(\phi,p_{\pm})}(z_2)=z_2(1+ T^{\pm\omega(\gamma_e)}z_1^{\pm 1})^{\mp 1}   
               \end{cases} \mbox{( mod $T^{\lambda}$)},              
            \end{align*} 
    
   \end{enumerate}
\end{thm} 
\begin{proof}
   The genericity $\vartheta$ guarantees that there is no $l_{\gamma}$ passing through $0$. By Gromov compactness theorem, there exists only finitely many $l_{\gamma}$ passing through a neighborhood of $0$ with the corresponding symplectic area is less than $\lambda$. 
   Therefore, there exists a neighborhood $\mathcal{U}_{\lambda}$ such that it avoids all such $l_{\gamma}$. 
   It is worth noticing that for any torus fibre $L$ over $U$, we have $H_2(X_U,L)\cong \mathbb{Z}$, which is generated by the Lefschetz thimble. Thus, the only holomorphic discs with small symplectic area have relative classes to be multiple cover of Lefschetz thimbles. In particular, this proves the first part of the theorem.   
   
   By Lemma 4.42 \cite{L4}, there exists a $1$-parameter family of hyperK\"ahler structures (which we will denote the path by $\phi$) connecting $X_U$ and $X_{OV}$(after hyperK\"ahler rotation), with the holomorphic volume form fixed. Moreover, the same underlying topological torus fibration is a special Lagrangian fibration with respect to any hyperK\"ahler structures in the $1$-parameter family. Assume that the size of the special Lagrangian fibre is small enough, so the the hyperK\"ahler structures of $X_{U}$ and $X_{OV}$ are closed enough to apply the Fukaya's trick (See Corollary 4.32 \cite{L4}). Without lose of generality, we may assume that the endpoints $u,u'$ of $\phi$ are closed enough so that the Fukaya's trick still apply by shrinking $U$ further if necessary.

    By Corollary \ref{49}, we may choose $\tilde{u},\tilde{u}'\in U$ near $u,u'$ such that $(X,L_{\tilde{u}})$ and $(X,L_{\tilde{u}'})$ do not bound any holomorphic discs when the complex structure changing via the path $\phi_1$. It worth mentioning that such locus only depends on the holomorphic volume form of the elliptic fibration. Therefore, we have $F^{can}_{\phi_1}=\mbox{Id}$. Choose a path $\phi_2$ (and $\phi_2'$) in $U$ connecting $u$ and $\tilde{u}$ ($u'$ and $\tilde{u}'$ respectively) which do not intersect $l_{\pm}$ (See Figure \ref{fig:55} below). 
     \begin{figure}
                 \begin{center}
                 \includegraphics[height=3in,width=6in]{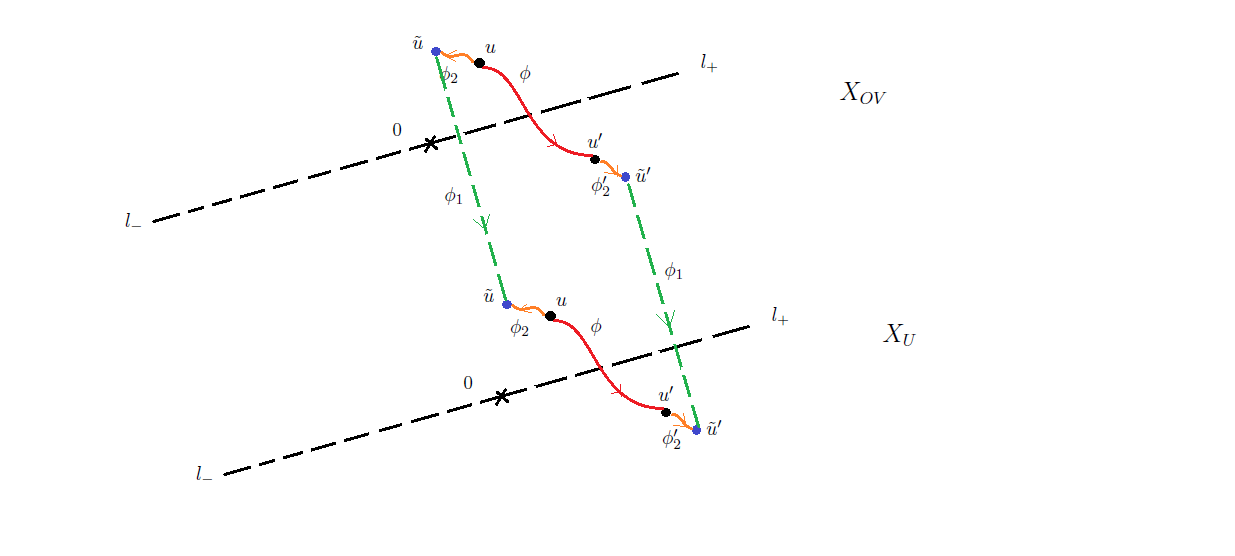}
                 \caption{Connecting $X_U$ to $X_{OV}$ via hyperK\"ahler manifolds.}
                  \label{fig:55}
                 \end{center}
                 \end{figure}
    
    By the construction of $U$, there is only holomorphic discs in the relative class of Lefschetz thimble (and its multiple covers), ending on fibres in $X_U$ and with symplectic area less than $\lambda$. Therefore, we have 
       \begin{align*}
          & F^{can}_{\phi_2}\equiv \mbox{Id} \mbox{( mod $T^{\lambda}$)} \\
          & F^{can}_{\phi_2'}\equiv \mbox{Id} \mbox{( mod $T^{\lambda}$)}.
       \end{align*}
   from the first part of the theorem. Therefore, 
      \begin{align} \label{46}
        F^{can}_{\phi}&= (F^{can}_{\phi'_2})^{-1}\circ(F^{can}_{\phi_1})^{-1} \circ F^{can}_{\phi_{OV}} \circ F^{can}_{\phi_1} \circ F^{can}_{\phi_2} \notag\\
                 &\equiv F^{can}_{\phi_{OV}} \mbox{( mod $T^{\lambda}$)}.
      \end{align} 
      Equation (\ref{46}) together with Theorem \ref{42} proves the second part of the theorem. 
\end{proof}
  
  \begin{rmk}
   A priori, we can't exclude the situation that a tropical/holomorphic disk in the relative class of $k\gamma_e$ with stop in $u\in B_0$ near $0$ when $k$ is large. Thus, it might be too optimistic to expect a universal $\mathcal{U}$ such that the following holds for all $\lambda>0$, which corresponds to the analogue statement of Section 11.5 \cite{KS1}.
  \end{rmk}

    \subsection{Open Gromov-Witten Invariants Revisited} \label{4055}
    
The wall-crossing phenomenon of Maurer-Cartan elements motivates the definition of another open Gromov-Witten invariant $\tilde{\Omega}'(\gamma;u)$, which satisfies the Kontsevich-Soibelman wall-crossing formula and shares the most of the properties with $\tilde{\Omega}(\gamma;u)$ defined in \cite{L4}. 
  
  \begin{lem} \label{3001} Assume that $X$ satisfies the generic condition $(*)_{\lambda}$. Assume that $\gamma$ is a primitive relative class such that $|Z_{\gamma}(u)|<\lambda$ and $u\notin l_{\gamma}$. Then there exists a small enough neighborhood $U_{\lambda}$ of $u$ such that for $u'_{\pm},u''_{\pm}\in U_{\lambda}$ on different sides of $l_{\gamma}$ and  $\mbox{Arg}Z_{\gamma}(u'_{\pm})<\mbox{Arg}Z_{\gamma}(u''_{\pm})$. Let paths $\phi_{\pm}\in U_{\lambda}$ such that $\phi_{\pm}(0)=u'_{\pm}$ and $\phi_{\pm}(0)=u''_{\pm}$, then we have 
        \begin{align*}
            F^{can}_{(\phi_+,u)}=F^{can}_{(\phi_-,u)} \hspace{4mm}(\mbox{mod }T^{\lambda}).
        \end{align*}
  \end{lem}
  \begin{proof}
       From the proof of Theorem \ref{101}, the transformations $F^{can}_{\phi_{\pm}}$ act as 
         \begin{align*}
            F^{can}_{\phi_{\pm}}:z^{\partial \gamma'}\mapsto z^{\partial \gamma'} f_{\gamma,\pm}^{\langle\gamma',\gamma\rangle},
         \end{align*} for some function $f_{\gamma,\pm}\in 1+\Lambda^{\mathcal{S}}/F^{\lambda}\Lambda^{\mathcal{S}}[[z^{\partial\gamma}]]$. Let $\phi'$( and $\phi''$) be a path connecting from $u'_+$ to $u'_-$( and from $u''_+$ to $u''_-$ respectively) in $U_{\lambda}$ as shown in the picture below. By shrinking $U_{\lambda}$ if necessary, we may assume that there are only finitely many affine lines $\lambda_{\gamma'}$ in $U_{\lambda}$ such that the torus fibres above those affine lines can bound holomorphic discs i in relative class $\gamma'$ with $|Z_{\gamma'}(u')|<\lambda$ for some $u'\in U_{\lambda}\cap l_{\gamma'}$. Thus $F^{can}_{\phi'},F^{can}_{\phi''} $ (mod $T^{\lambda}$) can be expressed as product of elementary transformations. Let $d\in \mathbb{N}$ be the smallest integer such that 
        \begin{align*}
          f_{\gamma,+}\neq f_{\gamma,-}(\mbox{ mod } z^{d\gamma}).
        \end{align*} Now from Theorem \ref{101}, each term in 
        \begin{align} \label{349}
                   (F^{can}_{\phi'})^{-1}\circ F^{can}_{\phi_-}\circ F^{can}_{\phi''}   
                   \hspace{5mm} (\mbox{mod } T^{\lambda})
                \end{align} can be expressed in terms of the product of elementary transformations (see Remark \ref{348}).        
        Then from Baker-Campbell-Hausdorff formula, the coefficients of $T^{lZ_{\gamma}(u)}z^{l\partial\gamma}$, $l=1,\cdots d$, in (\ref{349}) is the same as the one in $F^{can}_{\phi_+}$, then there exist terms in $F^{can}_{\phi'}$ and $F^{can}_{\phi''}$ such that the product gives the term of the form $cz^{d\partial \gamma}$. Therefore, there exist $\gamma_i$ can be represented as holomorphic discs and $\gamma=\sum_i\gamma_i$, with at least one $\gamma_i$ such that $\langle \gamma_i,\gamma\rangle \neq 0$. If any of $l_{\gamma_i}$ does not pass through $u$, then we may avoid the situation by shrinking $\mathcal{U}_{\lambda}$. Otherwise, $u\in W'_{\gamma}$ and contradicts with the assumption.  
      
  \end{proof}
  
   \begin{center}
   \begin{figure}
   \begin{overpic}[scale=0.5,unit=1mm]
   {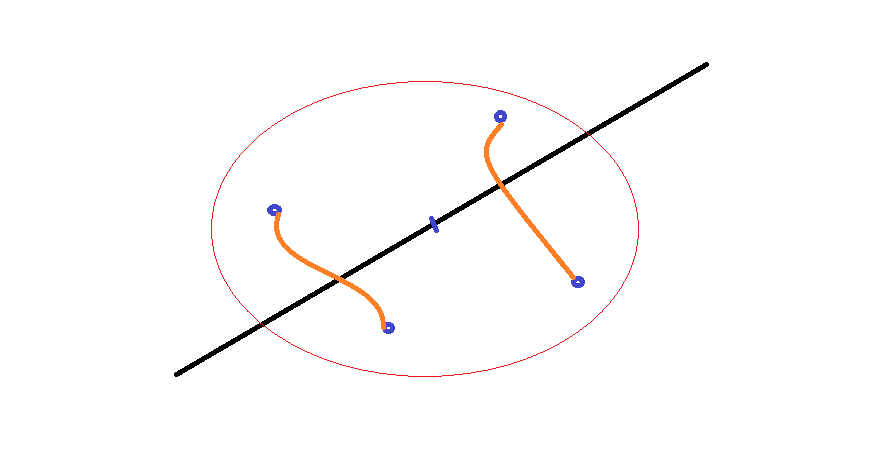}
   \put(60,28){\small \color{blue} $p$} 
   \put(32,35){\small \color{blue} $u'_+$}
   \put(59,42){\small \color{blue} $u'_-$}
   \put(46,13){\small \color{blue} $u''_+$}
   \put(78,24){\small \color{blue} $u''_-$}
   \put(34,24){\small \color{Orange}$\phi_+$}
    \put(70,33){\small \color{Orange}$\phi_-$}   
    \put(21,6){\small \color{black}$l$}   
   \end{overpic}
    \end{figure}               
    \end{center}              
  Let $\gamma$ be a primitive relative class, $l_{\gamma}$ an affine line  labeled by $\gamma$ and an unbounded increasing sequence $\lambda_i\in \mathbb{R}_{>0}$, there exits $u_{i,+},u_{i,-}\in U_{\lambda}$ on different sides of $l$ and a path $\phi_i$ in $U_{\lambda}$ connecting  $u_{i,+},u_{i,-}$ such that $\int_{\gamma}(\mbox{Im}\Omega+i\omega)$ is increasing, then the transformation $F^{can}_{\phi}$ acts as
    \begin{align*} 
       F^{can}_{\phi}:z^{\partial\gamma'} \mapsto z^{\partial\gamma'}f_{i,\gamma}^{\langle \gamma', \gamma\rangle},
    \end{align*} with $f_{i,\gamma}\in \Lambda^{\mathcal{S}_{\gamma}}[[z^{\partial \gamma}]]/F^{\lambda_i}\Lambda^{\mathcal{S}_{\gamma}}$. 
    \begin{definition} \label{4001}\begin{enumerate} 
      Assume that $X$ satisfies the generic condition (*)\footnote{One may replace $\Lambda[[H_1(L_u)]]$ by $\Lambda[[H_2(X,L_u)]]$ and all the arguments still go through.}. 
        \item With the notation above, the inverse limit $\underset{\leftarrow}{\lim}f_{i,\gamma}$ exists and is independent of various choices made in the construction. In other words, we can associate a power series 
                  \begin{align*}
                     f_{\gamma}(u):=\lim_{\leftarrow}f_{i,\gamma} \in \Lambda^{\mathcal{S}_{\gamma}}/F^{\lambda_i}\Lambda^{\mathcal{S}}[[z^{\partial\gamma}]]
                  \end{align*} for each $u\in l_{\gamma}\backslash \cup_{d\in \mathbb{N}}W_{d\gamma}'$. 
        \item Given $u\in B_0$ and $\gamma\in H_2(X,L_u)$. We will compute $f_{\gamma}(u)$ in $X_{\vartheta}$, $\vartheta=\mbox{Arg}Z_{\gamma}(u)-\frac{\pi}{2}$, and define the open Gromov-Witten invariants $\tilde{\Omega}'(\gamma;u)$
             by the formula 
              \begin{align} \label{4030}
                 f_{\gamma}(u)= \sum_{d\geq 1} d\tilde{\Omega}'(d\gamma;u)(z^{\partial \gamma}T^{Z_{\gamma}(u)})^d.
              \end{align}
    \end{enumerate}
    
    \end{definition}   
    Motivated by the Gross-Siebert program, we have the following conjecture:
    \begin{conj}
       The coefficients of $f_{\gamma}(u)$ are always positive. 
    \end{conj}
      
       A direct consequence of Lemma \ref{3001} is the following:
                       \begin{lem}    \label{3003}
                         Assume that the affine line segment from $u_1$ to $u_2$ does not intersect $W'_{d\gamma}$ for any $d\in \mathbb{Z}$ and $d|Z_{\gamma}|<\lambda$ on the affine segment from $u_1$ to $u_2$. Then we have 
                         \begin{align}
                             T_{u_1,u_2}f_{\gamma}(u_1)=f_{\gamma}(u_2) \hspace{4mm}(\mbox{mod }z^{d\partial\gamma}).
                         \end{align} 
                   \end{lem}

    In particular, this implies the following basic property of the open Gromov-Witten invariants $\tilde{\Omega}'(\gamma;u)$.
    \begin{prop} \label{3009}
       Assume that $X$ satisfies the generic condition $(*)_{\lambda}$, $u\notin W'_{\gamma}$ and $|Z_{\gamma}(u)|<\lambda$, then $\tilde{\Omega}'(\gamma;u)$ is well-defined. Moreover, we have the following properties:
       \begin{enumerate}
          \item If there exists a path connecting $u,u'\in B_0$ which does not pass through $W'_{\gamma}$, then 
                    \begin{align*}
                        \tilde{\Omega}'(\gamma;u)=\tilde{\Omega}'(\gamma;u').
                    \end{align*} 
          \item Let $u\in W'_{\gamma}$ generic and let $u_+,u_-$ be on the different sides of $W'_{\gamma}$ near $u$. Assume that there exists no $\gamma_i\in H_2(X,L_u)$ such that $\sum_{i}\gamma_i=\gamma$, $\mbox{Arg}Z_{\gamma_i}(u)=\mbox{Arg}Z_{\gamma}(u)$ and $\tilde{\Omega}'(\gamma;u)\neq 0$ with $\langle \gamma_i,\gamma_{i'}\rangle$ for some $i,i'$, then 
          \begin{align*}
              \tilde{\Omega}'(\gamma;u_+)=\tilde{\Omega}'(\gamma;u_-).
          \end{align*}        
          \item (Reality condition) 
          $\tilde{\Omega}'(\gamma;u)=\tilde{\Omega}'(-\gamma;u_-)$.
       \end{enumerate}
        \end{prop}
    \begin{proof} 
       If $\mbox{Arg}Z_{\gamma}(u)=\mbox{Arg}Z_{\gamma}(u')$, then the first property follows from Lemma \ref{3003} by dividing the path connecting $u,u'\in B_0$ into segments small enough. The proof of the general case is similar to the proof of the second property and we will skip it here.
       
        Fix $\lambda>|Z_{\gamma}(u)|$. From the assumption that $u\in W'_{\gamma}$ generic, we may assume that $u$ doesn't fall on $W'_{\gamma'}$ with $|Z_{\gamma'}(u)|<|Z_{\gamma}(u)|$. Choose a small counterclockwise contractible loop $\phi$ avoiding such walls except $W'_{\gamma}$ and apply Theorem \ref{102} with base point $u$. Then we have 
         \begin{align}\label{9005}
            \prod \tilde{\theta}_{\gamma_i}(u_+)= \prod \tilde{\theta}_{\gamma_i}(u_-) \hspace{5mm} (\mbox{ mod $T^{\lambda}$}).
         \end{align} On the other hand, each factor $\tilde{\theta}_{\gamma_i}$ are of the form $\exp{\tilde{f}_{\gamma_i}}$ from Remark \ref{348}, where $\tilde{f}_{\gamma_i}\in z^{\partial \gamma_i}\Lambda[[z^{\partial \gamma_i}]]$ and $f_{\gamma_i}(u), \tilde{f}_{\gamma_i}$ determines each other uniquely. The second property follows from Baker-Campbell-Hausdorff formula and compare the coefficient of $z^{\partial \gamma}$ on both sides of (\ref{9005}).

 The third property follows from checking the orientation of the moduli spaces of $\mathcal{M}_{k,\gamma}(\mathfrak{X},L_u)$ and $\mathcal{M}_{k,-\gamma}(\mathfrak{X},L_u)$, which is similar to the proof of the reality condition in \cite{L4}. 
       
    \end{proof}   
    Now we are ready for the weak correspondence theorem
    \begin{thm}\label{3021}
       Assume that $X$ satisfies the generic condition $(*)_{\lambda}$. Let $\gamma\in H_2(X,L_u)$ such that $\tilde{\Omega}'(\gamma;u)\neq 0$ and $Z_{\gamma}(u)|<\lambda$ . Then there exists a tropical disc $(\phi,T,w)$ with stop at $u$ and $[\phi]=\gamma$. 
    \end{thm}
    \begin{proof}
       Given $\gamma\in H_2(X,L_u)$ and $\tilde{\Omega}'(\gamma;u)\neq 0$, this implies that $\tilde{\Omega}'(\gamma;u')\neq 0$ for $u'$ in a neighborhood of $u$. Let $u'$ move along $l_{\gamma}$ in the direction such that $|Z_{\gamma}(u')|$ is decreasing. 
       \begin{lem}
       The function $|Z_{\gamma}|$ will decrease to zero on $l_{\gamma}$ at a point $u_0$. 
       \end{lem}
       \begin{proof}
          Assume that $|Z_{\gamma}|$ reaches a local minimal at $u_0$. Then by maximal principle for $Z_{\gamma}$ (or Lemma 2.7 \cite{L4}) we have either $Z_{\gamma}(u_0)=0$ or $u_0\in \Delta$. The generic choice of $\theta$ excludes the second situation. 
       \end{proof}
        If $\tilde{\Omega}'(\gamma;u')$ is constant on $l_{\gamma}$ between $u$ and $u'$, then $L_u'$ is a singular fibre. In particular, if $X$ has only $I_1$-type singular fibre then $\gamma$ is a multiple of Lefschetz thimble. If $\tilde{\Omega}'(\gamma;u')$ jumps at some point $u_1$ on $l_{\gamma}$, then $u'\in W'_{\gamma}$. From the second part of Proposition \ref{3009}, there exists $\gamma_i$ such that $\gamma=\sum_i \gamma_i$, $\tilde{\Omega}'(\gamma_i;u_1)\neq 0$ and $\langle \gamma_i,\gamma_{i'}\rangle \neq 0$ for some $i,i'$. Let 
         \begin{align*}
            l_{\gamma_i}^-:=\{u'\in l_{\gamma_i}|u' \mbox{ near $u_1$ and } |Z_{\gamma_i}(u')|<|Z_{\gamma_i}(u_1)|\}.
         \end{align*} Then union of $l_{\gamma}^-$ and the segment of $l_{\gamma}$ between $u$ and $u_1$ gives an image of tropical rational curve around $u_1$. Replace $\gamma_i$ by $\gamma$ and repeat the argument. Notice that $|Z_{\gamma_i}(u_1)|<|Z_{\gamma}(u_1)|$. When the symplectic area of holomorphic discs is small enough, then the speical Lagrangian torus fibre should be in a small tubular neighborhood of a singular fibre by Proposition 4.4.1 \cite{S6}. In particular, it implies that the relative homology class of the holomorphic discs can only be the multiple of the corresponding Lefschetz thimble. The procedure stops after finitely many steps by Gromov compactness theorem. The union of all local tropical rational curves gives an image of a tropical disc ends at $u$. Notice that the decomposition $\gamma=\sum_i \gamma_i$ at each vertex might not be unique and might lead to different (images of) tropical discs. 
    \end{proof}
    \begin{rmk}
       Comparing to the split attactor flows of black holes studied by Denef-Moore \cite{DM}, here the singularities of type $I_1$ correspond to the conifold points while we don't have the attactor points. 
    \end{rmk}
    Now we want to prove that the open Gromov-Witten invariants $\{\tilde{\Omega}'(\gamma;u)\}$ satisfies the Kontsevich-Soibelman wall-crossing formula.

    \begin{thm} \label{9001}
       Assume that $X$ satisfies the generic condition (*), then the open Gromov-Witten invariants $\tilde{\Omega}'(\gamma;u)$ satisfy the analogue of  Theorem \ref{3008}
    \end{thm}
    \begin{proof}
       Theorem \ref{102} is the Floer theoretic analogue of Theorem \ref{3011}. The rest of the proof is similar to the proof of Theorem \ref{3008}. 
    \end{proof}

     From the discussion in previous sections, the two open Gromov-Witten invariants $\tilde{\Omega}(\gamma;u)$ and $\tilde{\Omega}'(\gamma;u)$ share many properties. Actually, they are indeed the same invariants via the analogue of divisor axiom for Gromov-Witten invariants. 
  \begin{thm} \label{3005} Assume that $X$ satisfies the generic condition $(*)_{u,\gamma}$. Let $\gamma\in H_2(X,L_u)$ and 
     assume that $u\notin W'_{\gamma}$, then $\tilde{\Omega}(\gamma;u)$ and $\tilde{\Omega}'(\gamma;u)$ are both well-defined and 
       \begin{align*}
          \tilde{\Omega}(\gamma;u)=\tilde{\Omega}'(\gamma;u).
       \end{align*}
  \end{thm}
  We will postpone the proof to Section \ref{4005}. Similar to the tropical story, we have the analogue definition of wall of marginal stability.
      \begin{definition} 
         The wall of marginal stability for holomorphic discs are defined by
         \begin{align*}
            W_{\gamma}:=\{u\in W'_{\gamma}| \tilde{\Omega}(\gamma) \mbox{ differs on different sides of $W'_{\gamma}$}\}.
         \end{align*} For $u\in W'_{\gamma}\backslash W_{\gamma}$, we will define $\tilde{\Omega}(\gamma;u)$ by the natural continuous extension. 
      \end{definition}
      \begin{rmk}
         Similar proofs shows that Theorem \ref{3005} holds when $u\notin W_{\gamma}$. 
      \end{rmk}

  This will lead us to the main theorem of the paper, which matches the tropical discs counting with the open Gromov-Witten invariants.
  \begin{thm} \label{3013}
     Assume that $X$ satisfies the generic condition (*), $u\notin W_{\gamma}$ for $\gamma\in H_2(X,L_u)$
     , then
        \begin{align} \label{3022}
            \tilde{\Omega}(\gamma;u)=\tilde{\Omega}^{trop}(\gamma;u). 
        \end{align}
  \end{thm}
  \begin{proof}
     From Theorem \ref{3005}, it suffices to prove that 
           \begin{align*}
              \tilde{\Omega}'(\gamma;u)=\tilde{\Omega}^{trop}(\gamma;u),
           \end{align*} for $X$ satisfies generic condition $(*)_{\lambda}$ such that $\lambda>|Z_{\gamma}(u)|$. 
   It suffices to consider the situation that $u\notin W_{\gamma}\cup W^{trop}_{\gamma}$, since we extend $\tilde{\Omega}'(\gamma;u)$ and $\tilde{\Omega}^{trop}(\gamma;u)$ by natural continuous extension.  
  We will further assume that $u$ is generic in the sense that
 \begin{enumerate}
 \item All tropical discs contributing to $\tilde{\Omega}^{trop}(\gamma;u)$ have no internal vertex mapped to singularities.
 \item All tropical discs contributing to $\tilde{\Omega}^{trop}(\gamma;u)$ do not have images fall into the intersection of different walls of marginal stability. 
 \end{enumerate} This is true for tropical discs with respect to generic $\vartheta$ and we will choose $u$ generic such that $\mbox{Arg}Z_{\gamma}(u)$ is generic.
  Assume that the theorem holds for $u'$ near the general $u\notin W_{\gamma}$ with above properties, then we have 
            \begin{align*}
               \tilde{\Omega}(\gamma;u)=\tilde{\Omega}(\gamma;u')=\tilde{\Omega}^{trop}(\gamma;u')=\tilde{\Omega}^{trop}(\gamma;u),
            \end{align*} where the first equality is from Proposition \ref{3009} and the last one is from Theorem \ref{3037}. Therefore, the proof for general $u$ reduce to the case when $u$ is generic in the sense above. 
   
    Now for a general $u\notin W'_{\gamma}\cup W'^{trop}_{\gamma}$, we associate an integer $n(\gamma;u)$ to be the maximal number of internal vertices of the image of tropical discs stop at $u$ such that its relative class is $\gamma$. Explicitly,
       \begin{align*}
         n(\gamma;u):=\max \{|(C_0^{int}(T)|-1|(\phi,T,w) \mbox{ a tropical disc stop at $u$ with $[\phi]=\gamma$.}\}
       \end{align*} From Lemma \ref{3006}, the number $n(\gamma;u)$ is finite. We will first prove the theorem by induction on $n(\gamma;u)$. 
   
     First we assume that $n(\gamma;u)=0$. From the proof of Theorem \ref{3021}, we have that $\gamma$ is the parallel transport of the multiple of Lefschetz thimble from an $I_1$-singularity and 
         \begin{align*}
           \tilde{\Omega}'(\gamma;u)=\tilde{\Omega}^{trop}(\gamma;u),
         \end{align*} for any $u$ close enough to the singularity
        from Definition \ref{1011} and Theorem  \ref{3020}. 
     
     Now assume that the theorem is true for all pairs $(u',\gamma')$ with $n(\gamma';u')\leq n$ and now we want to prove the statement when $n(u;\gamma)=n+1$. There exists an affine line $l_{\gamma}$ passing through $u$ such that $Z_{\gamma}$ has constant phase along $l_{\gamma}$. Move $u$ along $l_{\gamma}$ in the direction such that $|Z_{\gamma}|$ is decreasing. From the assumption $l_{\gamma}$ will intersect $W^{trop}_{\gamma}$, say at $u_1$, the first point where $\tilde{\Omega}^{trop}(\gamma)$ jumps along $l_{\gamma}$. By Definition \ref{1012}, the tropical discs contributing to the jump of  $\tilde{\Omega}^{trop}(\gamma;u)$ on different sides of $W'^{trop}_{\gamma}$ are gluing of tropical discs with stop at $u_1$ together with the affine segment from $u_1$ to $u$ along $l_{\gamma}$. Let $\gamma_1,\cdots, \gamma_k$ be the relative classes of these tropical discs. Then $\tilde{\Omega}^{trop}(\gamma_i;u_1)\neq 0$ by 
          By Proposition 3.26 \cite{L4}, the tropical discs with stop at $u_1$ has the edges adjacent to $u_1$ mapped to the side of the wall with
        \begin{align} \label{444}  
         \frac{\mbox{Im}(Z_{\gamma_i}(u_1)\bar{Z}_{\gamma_{i'}}(u_1))}{\langle \gamma_i,\gamma_{i'}\rangle }<0
         \end{align}
          if $\langle \gamma_i,\gamma_{i'}\rangle\neq 0$. By induction hypothesis, we have $\tilde{\Omega}(\gamma_i;u_1)=\tilde{\Omega}^{trop}(\gamma_i;u_1)\neq 0$. In particular, the moduli spaces  $\mathcal{M}_{\gamma_i}(X,L_{u_1})$ are non-empty and $u_1\in W'_{\gamma}$.
           
       From Theorem \ref{9001}, there exists a neighborhood $\mathcal{U}\ni u_1$ such that for generic $u_{\pm}\in \mathcal{U}$, we have 
          \begin{align} \label{9003}
            T_{u_+,u_1} \prod^{\curvearrowright}_{\gamma\in H_2(X,L_u) \mbox{ primitive}: \atop \mbox{Arg}Z_{\gamma}(u)\in \mathcal{S} } \tilde{\theta}_{\gamma,<\lambda}(u_+) = T_{u_-,u_1}\prod^{\curvearrowright}_{\gamma\in H_2(X,L_u) \mbox{ primitive}: \atop \mbox{Arg}Z_{\gamma}(u)\in \mathcal{S} } \tilde{\theta}_{\gamma,<\lambda}(u_-)
          \end{align} where we take $\mathcal{S}$ be the open sector containing $\mbox{Arg}Z_{\gamma_i}$ for all $i$. The condition (\ref{444}) ensures that $\mathcal{S}$ can be chosen with angle less than $\pi$. At the same time, Theorem \ref{3008} gives the analogue equality for tropical discs counting. We choose $u_{\pm}$ on different sides of the wall such that $\frac{\mbox{Im}(Z_{\gamma_i}(u_+)\bar{Z}_{[\gamma_{i'}]}(u_+))}{\langle \gamma_i,\gamma_{i'}\rangle }>0$. All the relative classes appearing in the side of $u_-$ in (\ref{9003}) appears in the tropical analogue as well from the weak correspondence theorem (Theorem \ref{3021}). Conversely, all the relative classes appearing the the tropical analogue in the side of $u_-$ appear in the side of $u_-$ in (\ref{9003}) by induction hypothesis. In other words, the factors (both relative classes and the corresponding slab functions) appear in the side of $u_-$ in (\ref{9003}) and the tropical analogue are identical. The Kontsevich-Soibelman lemma (Theorem 6 \cite{KS1}) says that each factor in the side of $u_+$ in (\ref{9003}) are uniquely determined if the product is respect to the phase of the central charge. In particular, we have  $\tilde{\Omega}(\gamma;u'_1)=\tilde{\Omega}^{trop}(\gamma;u'_1)$, for $u'_1$ close enough to $u_1$ on $l_{\gamma}$ between $u$ and $u_1$. 
     
     From the definition of $u_1$, we have $\tilde{\Omega}^{trop}(\gamma;u)=\tilde{\Omega}^{trop}(\gamma;u'_1)$ and we now want to show that $\tilde{\Omega}(\gamma)$ does not jump between $u'_1$ and $u$ along $l_{\gamma}$. Assume that the open Gromov-Witten invariant jumps at $u_2$ (between $u'_1$ and $u$ on $l_{\gamma}$),
     then there exist ${\gamma}'_i\in H_2(X,L_{u_2})$ such that $\gamma=\sum_i{\gamma}'_i$, $\tilde{\Omega}(\gamma'_i;u_2)\neq 0$, $\mbox{Arg}Z_{\gamma'_i}(u_2)=\mbox{Arg}Z_{\gamma}(u_2)$ and $\langle \gamma'_i,\gamma'_{i'}\rangle \neq 0$ for some $i,i'$. 
     From Theorem \ref{3021}, we have $u_2\in W'^{trop}_{\gamma}$ and 
      $n(\gamma'_i;u_2)<n(\gamma;u)$. Therefore, we have $\tilde{\Omega}(\gamma'_i;u_2)=\tilde{\Omega}^{trop}(\gamma'_i;u_2)\neq 0$ by the induction hypothesis. Similar argument above shows that $\tilde{\Omega}^{trop}(\gamma)$ also jumps at $u_2$ and contradicts to the definition of $u_1$. Therefore, we have $\tilde{\Omega}(\gamma;u)=\tilde{\Omega}(\gamma;u'_1)$ and the theorem is proved. It worth noticing that we also prove that the wall of marginal stability for holomorphic discs and tropical discs are the same, i.e, $ W_{\gamma}=W^{trop}_{\gamma}$.
     
     \end{proof}
    
     \begin{thm}
       Given any elliptic K3 surface with only $I_1$-type singular fibres, the correspondence theorem (\ref{3022}) holds. 
     \end{thm}
     
     \begin{proof}
       Let $X$ be a general K3 surface with elliptic fibration and only type $I_1$ singular fibres. Fix $\lambda>|Z_{\gamma}(u)|$. By Theorem 4.1 \cite{L10} and Remark \ref{445}, there exists a $1$-parameter family of elliptic K3 surfaces $X_t$ with only typer $I_1$ singular fibres and such that $X_t$ satisfies the generic condition $(*)_{\lambda}$, for $t>0$. We may choose a trivializaton which induces diffeomorphism $X_t\cong X_0$. Choose a smooth section on $X_0$ induces a section of $X_t$ for $t\ll 1$ and the diffeomorphism of the base of elliptic fibraton on $X_t$ with the base of the elliptic fibration on $X_0$. We will use $W'^t_{\gamma}$ to denote the corresponding wall labeled by $\gamma$ in $X_t$ and $\tilde{\Omega}_t(\gamma;u)$ (and $\tilde{\Omega}_t^{trop}(\gamma;u)$) to denote the open Gromov-Witten invariants (and tropical disc counting invariants) of $X_t$. By Theorem \ref{3013}, we have the correspondence theorem $\tilde{\Omega}'_t(\gamma;u)=\tilde{\Omega}^{trop}_t(\gamma;u)$ for $X_t$, $t>0$. 
          
 Recall that the open Gromov-Witten invariant $\tilde{\Omega}(\gamma;u)$ is defined for $u\notin W'_{\gamma}$ and extend continuously to $u\notin W_{\gamma}$ for a general K3 surface $X$ with elliptic fibration. Notice that if $u\notin W'_{\gamma}$, then $u\notin W'^t_{\gamma}$ for $t\ll 1$. Thus, we have $\tilde{\Omega}_t(\gamma;u)$ is well-defined and $\tilde{\Omega}(\gamma;u)=\tilde{\Omega}_t(\gamma;u)$. Then we have 
      \begin{align*}
        \tilde{\Omega}(\gamma;u)=\tilde{\Omega}_t(\gamma;u)=\tilde{\Omega}_t^{trop}(\gamma;u)=\tilde{\Omega}^{trop}(\gamma;u).
      \end{align*} Here the last equality is from Theorem \ref{9011} and this finishes the prove of the theorem. 
    
  \end{proof}

\subsection{Open Gromov-Witten Invariants via Real Noether-Lefschetz Theory} \label{4005}
 The Lemma \ref{3001} and Definition \ref{4001} can actually be generalized to the following setting following the same argument. Consider $\mathcal{M}$ be the moduli space of the marked K3 surfaces together with special Lagrangians. Let $\mathbb{L}_{K3}$ be the K3 lattice. Explicitly, $\mathcal{M}$ is the set of $5$-tuples,
      \begin{align}
     \mathcal{M}=\{(X,\omega,\Omega,\alpha,L)\}/\sim, 
       \end{align} where $X$ is a K3 surface with a hyperK\"ahler pair $(\omega,\Omega)$, $\alpha:H^2(X,\mathbb{Z})\cong \mathbb{L}_{K3}$ is a marking and $L$ is a smooth (oriented) special Lagrangian torus with respect to the pair $(\omega,\Omega)$, i.e. $\omega|_{L}=\mbox{Im}\Omega|_{L}=0$. Two $5$-tuples $(X,\omega,\Omega,\alpha,L)$ and $(X',\omega',\Omega',\alpha',L')$ are equivalent if and only if there exists a diffeomorphism $f:X\rightarrow X'$ such that $f^*\omega'=\omega$, $f^*\Omega'=\Omega$, $\alpha'\circ f^*=\alpha$ and $f(L)=L'$. The moduli space $\mathcal{M}$ is an infinite disjoint union of smooth oriented manifolds of real dimension $59$\cite{L6}. It worth mentioning that the boundary of one irreducible component of $\mathcal{M}$ may fall in another irreducible component.  
  
      We will abuse notation and denote $\mathcal{M}$ for one of its components and write $(X,L)$ for the $5$-tuple by view the rest as part of the data in $X$ for simplicity. Now assume that there is a pair $u=(X_0,\omega_0,L_0)\in \mathcal{M}$ such that $L_0$ bounds a holomorphic disc in the relative class $\gamma\in H_2(X_0,L_0)$. Then there is a submersion from a small neighborhood $\mathfrak{U}$ of $\mathcal{M}$ containing $(X_0,L_0)$ such that $\int_{\gamma_{(X,L)}}\omega_X\neq 0$, for $(X,L)\in \mathfrak{U}$. Thus there is a well-defined map from $\mathfrak{U}$ to
      $S^1_\vartheta$ given by,
      \begin{align*}
         \mbox{Arg}_{\gamma}=\mbox{Arg}(Z_{\gamma}):\mathfrak{U}\rightarrow S^1_{\vartheta}, 
      \end{align*} where 
      \begin{align*}
      Z_{\gamma}(X,L)=-i\int_{\gamma_{(X,L)}}\mbox{Im}\Omega_X+i\omega_X,
      \end{align*} for $(X,\omega,\Omega,\alpha, L)\in \mathfrak{U}$. The map $Arg_{\gamma}$ is a locally a submersion and thus the fibre of $\mbox{Arg}_{\gamma}$ over $1\in S^1$, denoted by $NL_{\gamma}$, is a real codimension one submanifold of $\mathcal{M}$. $NL_{\gamma}$ can be viewed as the real analogue of the Noether-Lefschetz
      divisor.  We will call them Noether-Lefschetz walls because they are codimension one submanifolds, which separates a tubular neighborhood of locus of $\mathcal{M}$ where $\gamma$ can be represented as holomorphic discs into two parts. Since the phase
      $S^1_{\vartheta}$ is viewed as the unit circle in the complex plane and thus is naturally oriented, the Noether-Lefschetz walls
      $NL_{\gamma}$ are also oriented.

  \begin{proof}(of Theorem \ref{3005})

     Now assume there are a sequence of pairs of points $\{u^{\pm}_i=(X^{\pm}_i,\omega^{\pm}_i,\Omega^{\pm}_i,\alpha^{\pm}_{i},L^{\pm}_i)\}$ in $\mathcal{M}$ converging to $u$ but fall on different sides of $NL_{\gamma}$. Choose paths $\phi_i$ connected $u^{\pm}$ in a neighborhood in $u$. Following the same argument in Lemma \ref{3001} and Definition \ref{4001}, we get a power series  $  f_{\gamma}(u)\in \Lambda[[z^{\partial \gamma}]]
      $, for $u\in NL_{\gamma}$. Same argument in Lemma \ref{3001} shows that $f_{\gamma}(u)$ is independent of the choice of the sequences $u^{\pm}_i$. In particular, we are interested in the following two cases (the paths both go in the direction such that $\mbox{Arg}Z_{\gamma}$ is increasing): 
     \begin{enumerate}
       \item $X^+_i=X^-_i, \omega^+_i=\omega^-_i,\Omega^+_i=\Omega^-_i, \alpha^+_i=\alpha^-_i$. $L^+_i$ and $L^-_i$ are two special Lagrangians on different sides of the affine line $l_{\gamma}$.  
       \item $X^+_i=X^-_i, \alpha^+_i=\alpha^-_i, L^+_i=L^-_i$, and 
          \begin{align*}
            &\omega^{\pm}_i= \mbox{Im}(e^{\mp i\epsilon_i}(\mbox{Im}\Omega_0+i\omega_0)), \\
            &\Omega^{\pm}_i=\mbox{Re}\Omega_0+i\mbox{Re}(e^{\mp i\epsilon_i}(\mbox{Im}\Omega_0+i\omega_0)),
          \end{align*}
       for some $\epsilon_i \searrow 0$. 
       In other words, the path $\phi_i$ is an arc contained the $S^1$-family of hyperK\"ahler structures described in Section \ref{4003} with two end points on the different side of $NL_{\gamma}$. 
     \end{enumerate} 
    Let $\phi_1$ is a path of the first kind from $u_1^-$ to $u_1^+$ and $\phi_2$ is a path of the second kind from $u_2^-$ to $u_2^+$ in a small neighborhood of $u$. Choose $\phi_+$ connects $u_1^+, u_2^+$ and $\phi_-$ connects $u_1^-,u_2^-$ such that the composition of the paths $\phi_1 \circ \phi_+ \circ \phi_2^{-1} \circ \phi_-^{-1}$ is contractible, since $\mathcal{M}$ is a manifold. Therefore, 
        \begin{align*}
            F^{can}_{(\phi_1,u)}=(F^{can}_{(\phi_+,u)})^{-1}\circ F^{can}_{(\phi_2,u)}\circ F^{can}_{(\phi_-,u)}. 
        \end{align*}
     Similar to the proof of Proposition \ref{3009} (where $\mathcal{M}$ plays the role of $B_0$ and $NL_{\gamma}$ is the analogue of $l_{\gamma}$), we can choose the path $\phi_{\pm}$ such that the coefficients of $T^{Z_{\gamma}(u)}z^{\partial\gamma}$ of the power series corresponding to $F^{can}_{(\phi_1,u)},F^{can}_{(\phi_2,u)}$ coincide since $u$ corresponds to a pair of K3 and special Lagrangian torus which does not fall on $NL_{\gamma}$ but not on $NL_{\gamma_1}\cap NL_{\gamma_2}$ for $\gamma_1+\gamma_2=\gamma$. By definition, the former defines the open Gromov-Witten invariant $\tilde{\Omega}'(\gamma;u)$. Let $e^*_1,e^*_2$ be an integral basis of $H_1(L_u)$ such that $\langle e^*_1,e^*_2\rangle=1$ and $\partial \gamma=ce^*_1, c\in \mathbb{N}$. To show the theorem, it suffices to show that 
        \begin{align*}
           -d\tilde{\Omega}(d\gamma;u)z_1^{cd}e_2=\sum_{k\geq 1} \mathfrak{c}'(k,d\gamma)(b),
        \end{align*} From the proof of Theorem \ref{101}, we have
        $\mathfrak{c}(k,d\gamma)(b,\cdots,b)=0$, for $k\geq 2$. For $k=1$, we have
            \begin{align*}
               \mathfrak{c}'(1,d\gamma)(b)= (z^{\partial \gamma})^d \int_0^1 c^{\tau}_{0,d\gamma}d\tau.
            \end{align*}
     By Lemma \ref{27}, we have 
         \begin{align*}
             \mathfrak{c}'(1,d\gamma)(b)&=(z^{\partial \gamma})^d \int_0^1 c^{\tau}_{0,d\gamma}d\tau \\
                &=-(z^{\partial \gamma})^d \tilde{\Omega}(d\gamma;u)de_2.
         \end{align*}
   \end{proof}


$ 
 $	
\appendix
\flushbottom
\include{AppendixA}
\flushbottom
\section{Existence of Quadratic Refinement} \label{2997}

\begin{prop} Let $\Gamma'\subset \Gamma$ be the subset of $\Gamma$ consisting of pair $(u,\gamma)$ such that $\gamma$ can be realized as tropical discs with stop at $u$ and with respect to some $\vartheta\in S^1$. Then there exists a unique well-defined quadratic refinement defined on $\Gamma'$ such that the value is $-1$ for every parallel transport of the Lefschetz thimble. 
\end{prop}
 \begin{proof}
    We will actually prove that there exists an integer-value function $C:\Gamma_g:=\cup_{u\in B_0}H_1(L_u) \rightarrow \mathbb{Z}$ satisfying 
      \begin{enumerate}
          \item For any $\partial \gamma_1,\partial \gamma_2\in \Gamma_g$, we have
           \begin{align*}
           C(\partial\gamma_1+\partial\gamma_2)=C(\partial\gamma_1)C(\partial\gamma_2)+\langle \partial\gamma_1,\partial \gamma_2 \rangle (\mbox{ mod }2).
           \end{align*}
          \item If $\partial\gamma\in \Gamma_g$ is a parallel transport of vanishing cycle from an $I_1$-type singular fibre then $C(\partial \gamma)=1$.
      \end{enumerate}
     Then any $\gamma\in \Gamma$, we can define 
           \begin{align*}
              c(\gamma)=(-1)^{C(\partial \gamma)}.
           \end{align*}
    
    Assume there are two tropical discs images represents the same relative homology class. Without lose of generality, we may assume that there exists integers $k_i, k\in \mathbb{Z}$ and parallel transport of Lefschetz thimbles $\gamma, \gamma_i \in H_2(X,L_u)$ such that 
         \begin{align*}
            \sum_i k_i\gamma_i=k\gamma.
         \end{align*}
    It suffices to prove that the quadratic refinement satisfies
          \begin{align} \label{9999}
             C(\sum_i k_i\partial\gamma_i)=C(k\partial\gamma) ( \mbox{ mod }2),
          \end{align} where $k_i, k\in \mathbb{Z}$ and $\partial \gamma_i,\partial\gamma\in H_1(L_u)$ are parallel of vanishing cycles.    
    Notice that $c(\partial\gamma_i)=1$ and $c(\partial\gamma)=1$.
    By choose a symplectic basis for $H_1(L_u)$ and set 
         \begin{align*}
            \partial\gamma_i=(v_i^1,v_i^2), \hspace{5mm} \partial \gamma=(v^1,v^2).
         \end{align*} Then we have 
         \begin{align*}
             c(k_i\partial\gamma_i)&=1-(1-k_iv_i^1)(1-k_iv_i^2) \\
             c(k_i\partial\gamma)&=1-(1-\sum_i k_iv_i^1)(1-\sum_i k_iv_i^2). 
         \end{align*} The equation (\ref{9999}) is reduced to the following identity in $\mathbb{Z}_2$
         \begin{align*}
            1-\big(1-\sum_i k_i v_i^1\big)\big(1-\sum_i k_i v_i^2  \big)=\sum_k \big[ 1-(1-v_k^1)(1-v_k^2)\big]+\sum_{i\neq j} v_i^1 v_j^2.            
         \end{align*}     
    
 \end{proof}
     
\include{AppendixB}
\flushbottom
\section{Refined Tropical Countings} \label{4099}

      The Kontsevich-Soibelman algebra can be $q$-deformed as follows \cite{KS2}: the algebra becomes non-commutative
        \begin{align*}
           z^{\partial\gamma_1}z^{\partial\gamma_2}=q^{\langle\gamma_1,\gamma_2\rangle}z^{\partial\gamma_1}z^{\partial\gamma_2},
        \end{align*} where $q$ is a formal variable. The Lie bracket is replaced by 
        \begin{align*}
           [z^{\partial\gamma_1},z^{\partial\gamma_2}]=(q^{n/2}-q^{-n/2})z^{\partial\gamma_1+\partial\gamma_2}. 
        \end{align*} 
        The $q$-deformed dilogarithm 
          \begin{align*}
             Li_2(z;q):=\sum^{\infty}_{k=1}\frac{z^k}{k(1-q^k)}
          \end{align*}
       induces a $q$-deformation of $\theta_{\gamma}$ by 
          \begin{align*}
             \theta_{q,\gamma}(u)=Ad\big(\frac{1}{d}Li_2(q^{1/2}T^{Z_{\gamma}(u)}z^{\partial\gamma})\big),
          \end{align*} where $d$ is the divisibility of $\partial\gamma$.
     With the above notation, we are ready for the definition of refined tropical discs counting invariants.
     \begin{definition}
        \begin{enumerate}
           \item Under the same notation as Definition \ref{1011}, we define its weight of an admissible tropical disc $(\phi,T,w)$ to be 
     \begin{align*}
                  \mbox{Mult}_q(\phi):=\prod_{v\in C^{int}_0(T)}[\mbox{Mult}_v(\phi)]_q\prod_{v\in C_0^{ext}(T)\backslash \{u\}}\frac{(-1)^{w_v-1}}{w_v[w_v]_q}  \prod_{\bar{v}\in C^{int}_0(\bar{T})}|\mbox{Aut}(\bold{w}_{\bar{v}})|.
              \end{align*} Here we use the notation 
                \begin{align*}
                        [n]_q:= \frac{q^{n/2}-q^{-n/2}}{q^{1/2}-q^{-1/2}}, \mbox{ for $n\in \mathbb{N}$} .
                    \end{align*}
        \item Let $u\in B_0$ and $\gamma\in H_2(X,L_u)$ such that $u\notin W'^{trop}_{\gamma}$. The refined tropical discs counting invariant  $\tilde{\Omega}_q^{trop}(\gamma;u)$ is defined to be 
                \begin{align*}
                   \tilde{\Omega}_q^{trop}(\gamma;u):=\sum_{\phi} \mbox{Mult}_q(\phi) ,
                \end{align*} where the sum is over all admissible tropical discs on $B$ with stop at $u$ such that $[\phi]=\gamma$.        
           \end{enumerate}        
                    
     \end{definition} 
      The definition of refined tropical discs counting again looks artificial. However, they satisfy the refined wall-crossing formula. 
       \begin{thm}
       Then the refined tropical discs counting invariants satisfy the refined wall-crossing formula.                   
       \end{thm}   
       \begin{proof}
         The proof is essentially the same as the proof of Theorem \ref{3008}. The refined analogue of Theorem \ref{3011} is given by Corollary 4.9 in \cite{FS}.
       \end{proof}

    Together with the Theorem \ref{3013} which states that
       \begin{align*}
         \tilde{\Omega}(\gamma;u)=\tilde{\Omega}^{trop}(\gamma;u),
       \end{align*} this suggests that the open Gromov-Witten invariants $\tilde{\Omega}(\gamma;u)$ also admits an refinement.        
       We will leave the geometric interpretation of the refined open Gormov-Witten invariants for future work.

Department of Mathematics, Columbia University\\
E-mail address: yslin@math.columbia.edu

\end{document}